\documentclass[12pt]{article}
\usepackage{amsmath, amssymb, amsfonts, latexsym, amsthm, amscd, epsfig, graphics, subfig}
\usepackage{hyperref}
\usepackage{lineno}

\newtheorem{theorem}{Theorem}[section]

\newtheorem{lemma}[theorem]{Lemma}

\newtheorem{remark}[theorem]{Remark}
\numberwithin{equation}{section}
\allowdisplaybreaks[4]

\newcommand*\patchAmsMathEnvironmentForLineno[1]{%
  \expandafter\let\csname old#1\expandafter\endcsname\csname #1\endcsname
  \expandafter\let\csname oldend#1\expandafter\endcsname\csname end#1\endcsname
  \renewenvironment{#1}%
     {\linenomath\csname old#1\endcsname}%
     {\csname oldend#1\endcsname\endlinenomath}}%
\newcommand*\patchBothAmsMathEnvironmentsForLineno[1]{%
  \patchAmsMathEnvironmentForLineno{#1}%
  \patchAmsMathEnvironmentForLineno{#1*}}%
\AtBeginDocument{%
\patchBothAmsMathEnvironmentsForLineno{equation}%
\patchBothAmsMathEnvironmentsForLineno{align}%
\patchBothAmsMathEnvironmentsForLineno{flalign}%
\patchBothAmsMathEnvironmentsForLineno{alignat}%
\patchBothAmsMathEnvironmentsForLineno{gather}%
\patchBothAmsMathEnvironmentsForLineno{multline}%
}

\title{\bf Evolving convex curves by a generalized length-preserving flow}

\author{\ {\bf Laiyuan Gao\thanks{The corresponding author.} ~~~~Shengliang Pan} }
\date{}

\begin{document}
\maketitle

\noindent {\bf Abstract} This paper deals with a generalized length-preserving flow for convex curves in the plane.
It is shown that the flow exists globally and deforms smooth initial curves into circles as time tends to infinity.

\baselineskip 15pt

\section{Introduction}
\setcounter{equation}{0}
Let $X(\phi, t):  S^1 \times [0, \omega)\rightarrow\mathbb{R}^2$ be a family of smooth and closed curves in the plane.
Denote by $\kappa(\phi, t)$ the (relative) curvature with respect to the Frenet frame $\{T(\phi, t), N(\phi, t)\}$ at a point
$(\phi, t)$, where $T(\phi, t)$ and $N(\phi, t)$ are the unit tangent and inward normal vectors, respectively.
An embedded closed $C^2$ curve is called convex if its curvature is positive everywhere.
Suppose $X_0(\phi)$ is a convex curve in the plane.
This paper focuses on a general nonlocal curvature flow for convex curves:
\begin{equation}\label{eq:1.1.201811}
\left\{\begin{array}{ll}
\frac{\partial X}{\partial t}(\phi, t) = \left[F(\kappa)-\frac{1}{2\pi}\int_0^L F(\kappa) \kappa ds\right] N(\phi, t)
     \ \ \textup{in}~ S^1\times(0, \omega),\\\\
X(\phi, 0) = X_0(\phi) \ \ \textup{on}~ S^1,
\end{array} \right.
\end{equation}
where $s$  represents the arc-length parameter, $L=L(t)$ the length of $X(\cdot, t)$ and $F(\cdot)$ a smooth
real-valued function on $(0, +\infty)$.

In recent decades, the curvature flow of curves has attracted extensive attention since 1980s. Besides the well known curve
shortening flow (see \cite{Gage-1983, Gage-1984}, \cite{Gage-Hamilton} and \cite{Grayson-1987})
and its generalizations (see \cite{Andrews-1998, Andrews-2003}), there are also nonlocal flows, such as length-preserving flows by
Pan-Yang \cite{Pan-Yang-2008}, Ma-Zhu \cite{Ma-Zhu-2012} and area-preserving flows by Gage \cite{Gage-1986}, Ma-Cheng \cite{Ma-Cheng-2014},
Mao-Pan-Wang \cite{Mao-Pan-Wang-2013}. Lin and Tsai \cite{Lin-Tsai-2012} summarised previews length-preserving
or area-preserving flows into the model \eqref{eq:1.1.201811}. Let $\alpha > 0$ be a constant.
They especially studied the asymptotic behavior of $\kappa^\alpha$-type and $1/\kappa^\alpha$-type nonlocal flows
as $t\rightarrow +\infty$, provided that the initial curve is convex and the flow exists globally.
Subsequently, Wang and Tsai \cite{Wang-Tsai-2015} demonstrated that $\kappa^\alpha$-type length-preserving or area-preserving flows for convex curves
exist globally and drive the evolving curve into circles as $t \rightarrow + \infty$.
The $\alpha$-homogeneity of $F$ plays an essential role in Wang-Tsai's research. The higher
dimensional similar model with $\alpha$-homogeneity assumption of $F$ has been treated by Andrews and McCoy \cite{Andrews-McCoy-2012}.
Recently, Gao, Pan and Tsai \cite{Gao-Pan-Tsai-2021, Gao-Pan-Tsai-2020-1, Gao-Pan-Tsai-2020-2} studied $1/\kappa^{\alpha}$ type nonlocal flows for convex curves.
A number of blow-up examples of this kind of nonlocal flows have been discovered, and all global flows are proven convergent as time tends to infinity.

Both the $\kappa^\alpha$-type and $1/\kappa^\alpha$-type nonlocal flows are special cases of the flow (\ref{eq:1.1.201811}). Thus, the research
of this generalized model is a continuation of previous studies. The flow (\ref{eq:1.1.201811}) with some concrete $F$ can be defined for general embedded curves,
but it may experience a blow-up in a finite time if the function $F$ or the initial convex curve is not appropriately chosen. One can find
some blow-up examples in \cite{Gage-1986, Mayer-2001} ($X_0$ is embedded but not convex) and in \cite{Gao-Pan-Tsai-2020-1, Tsai-2018} ($X_0$ is convex).
To obtain favorable results of the flow, including its global existence and convergence, certain restrictions on the function $F$ should be imposed.

From this point forward, we assume the function $F$, defined in the interval $(0, +\infty)$, is smooth and satisfies the following three conditions for $u>0$:\\
(i) $F$ is strictly increasing on $(0, +\infty)$, i.e. $F^\prime(u)>0$ and
      $$\lim\limits_{u\rightarrow 0^+} F^\prime (u)\cdot u =0;$$
(ii) $F(u)$ is positive for all $u>0$ and
$$
\lim\limits_{u\rightarrow +\infty} F(u) = +\infty, ~~~~F(0):=\lim\limits_{u\rightarrow 0^+} F(u) =0;
$$
(iii) $F(u)$ and $F^\prime(u)$ satisfy two asymptotic properties
\begin{eqnarray*}
\lim_{u\rightarrow +\infty} \frac{F^\prime (u)\cdot u^2}{F(u)}=+\infty, ~~~~
\lim_{u\rightarrow 0^+} \frac{F^\prime (u)\cdot u^2}{F(u)}=0.
\end{eqnarray*}

Before the statement of our main result, two remarks are in order.
First, in contrast to the evolution problem of hypersurfaces (see \cite{Andrews-1994, Andrews-2007, Gerhardt-1990}, etc.),
there are no additional convexity conditions on the function $F$.
Along with the $\alpha$-homogeneity case of $F$ studied by Wang-Tsai \cite{Wang-Tsai-2015},
there are numerous other examples of $F$ satisfying above conditions (i)-(iii), such as
$$F(\kappa)=\ln (1+\kappa), ~e^\kappa - 1, ~2\kappa+\sin(\kappa), ~\kappa^2 \ln \kappa +\kappa$$ and so on.
Second, the parabolic property $F^\prime > 0$ and the positivity condition $F > 0$ imply the limit given in the condition (i)
(See Remark \ref{rmk:3.1.201811}), if this limit exists.
Therefore, the conditions (i)-(iii) can be summarized as the parabolic property, positivity and five limits of $F$.
The main result of this paper is as follows.
\begin{theorem}\label{thm:1.1.201811}
Let $X_0$ be a smooth and convex plane curve. If $F$ satisfies
the above conditions (i)-(iii) then the flow (\ref{eq:1.1.201811}) exists on time interval $[0, +\infty)$,
keeps the convexity of the evolving curve, preserves its length and deforms $X(\cdot, t)$ into a finite circle
as time tends to infinity.
\end{theorem}


It is extremely challenging to obtain good convergence results of the flow (\ref{eq:1.1.201811}) for generic initial convex curves, if
any one of the three conditions of $F$ is omitted.
The condition (i) is used to ensure the parabolicity of the evolution equation.
In condition (ii), we assume $F$ is a positive function.
Tsai \cite{Tsai-2018} discovered the blow up phenomenon of a convex curve evolving under
(\ref{eq:1.1.201811}) when $F(\kappa)=-1/\kappa$.
The same phenomenon also occurs in some other similar nonlocal models with $F<0$ \cite{Gao-Pan-Tsai-2021, Gao-Pan-Tsai-2020-1, Gao-Pan-Tsai-2020-2}.
Hence, without condition (ii), this flow may blow up a convex curve in finite time.
The third condition can be used to uniformly bound the curvature both above and below.
The proof depends on a delicate application of the maximum principle for two geometric auxiliary functions,
which are inspired by Chou's early work \cite{Chou-1985} and its developments \cite{Andrews-McCoy-2012},
\cite{Gao-Zhang-2019} and \cite{Wang-Tsai-2015}.
Without the uniform bounds of curvature, the convergence of the flow cannot be expected.

Lin and Tsai in \cite{Lin-Tsai-2012} proposed that one might consider the flow (1.1) under the assumption that $F(u)/u$ is increasing if $u > 0$.
This condition implies the parabolicity condition $F^\prime > 0$ when $F(u) > 0$ for all positive $u$. It remains an open question whether or not Lin-Tsai's
condition leads to the convergence of the flow (\ref{eq:1.1.201811}).
For flows of other type of curves, one can consult Wang-Kong \cite{Wang-Kong-2014}, Wang-Wo \cite{Wang-Wo-2014}, Wang-Wo-Yang \cite{Wang-Wo-Yang-2018}.
One can read the books \cite{Chou-Zhu-2001, Mantegazza-2011, Zhu-2002}
and references therein for more information about the curvature flows of curves and some higher dimensional analogues.

The global existence of the flow (\ref{eq:1.1.201811}) is proven in Section \ref{sec:2.201811}.
The convergence of the curvature is demonstrated in Section \ref{sec:3.201811} and the convergence of the curve is presented in Section
\ref{sec:4.201811}.

\section{Global existence}\label{sec:2.201811}
In this section, the global existence of the flow (\ref{eq:1.1.201811}) is proven.
It is shown that the evolving curve keeps convex and its curvature has a uniform upper bound,
which is independent of time. So the flow can be extended on time interval $[0, +\infty)$.

\subsection{Short time existence}\label{sec:2.1.201811}
Let $\theta=\theta(\varphi, t)$ be the tangent angle of the evolving curve. Under the flow (\ref{eq:1.1.201811}),
$\theta$ varies with respect to time. One can make it independent of time by adding a proper tangent component to the
original flow to obtain
\begin{equation}\label{eq:2.1.201811}
\left\{\begin{array}{ll}
\frac{\partial \widetilde{X}}{\partial t} = -\frac{1}{\kappa}\frac{\partial F}{\partial s} T+ (F - \lambda) N
     \ \ \textup{in}~ S^1\times(0, \omega),\\\\
\widetilde{X}(\varphi, 0) = X_0(\varphi) \ \ \textup{on}~ S^1.
\end{array} \right.
\end{equation}
If there is a family of convex curves evolving under the flow (\ref{eq:1.1.201811}), then (\ref{eq:2.1.201811}) also has a
solution and it differs from $X(\cdot, t)$ by a reparameterization and a translation \cite{Chou-Zhu-2001}.
By straightforward computations (see for example, \cite{Chou-1985}), the tangent angle under the flow (\ref{eq:2.1.201811}) satisfies that
\begin{eqnarray*}
\frac{\partial \theta}{\partial t} = \left(-\frac{1}{\kappa} \frac{\partial F}{\partial s} \right) \kappa + \frac{\partial }{\partial s} (F- \lambda) =0.
\end{eqnarray*}
So $\theta$ is independent of time. It can be used as the parameter of the evolving curve $X(\cdot, t)$.
Due to this reason, we consider the flow (\ref{eq:2.1.201811}) from now on.

Since the curvature determines the curve up to a transformation via the Euclidean group, the short time existence of the flow (\ref{eq:2.1.201811})
can be reduced to the evolution equation of the curvature $\kappa$, i.e.,
\begin{eqnarray}\label{eq:2.2.201811}
\frac{\partial \kappa}{\partial t}= \kappa^2 \left[F^\prime (\kappa) \frac{\partial^2 \kappa}{\partial \theta^2}
         + F^{\prime\prime}(\kappa) \left(\frac{\partial \kappa}{\partial \theta}\right)^2
         + F(\kappa)-\frac{1}{2\pi}\int_0^{2\pi} F(\kappa) d\theta\right],
\end{eqnarray}
where the initial curvature $\kappa_0 (\theta)= \kappa (\theta, 0)$ satisfies the closing condition
\begin{eqnarray}\label{eq:2.3.201811}
\int_0^{2\pi} \frac{e^{i\theta}}{\kappa_0(\theta)} d\theta =0.
\end{eqnarray}

The equation (\ref{eq:2.2.201811}) is a quasilinear parabolic equation. Its linearization at a smooth function $u(\theta, t) = \kappa_0(\theta) + tu_0(\theta)$ is as follows
\begin{eqnarray*}
\frac{\partial \kappa}{\partial t} &=& u^2 F^\prime (u)\frac{\partial^2 \kappa}{\partial \theta^2}
   + 2u^2 F^{\prime\prime}(u) \frac{\partial u}{\partial \theta}\frac{\partial \kappa}{\partial \theta}
\\
&&+2u\left[F^\prime (u) \frac{\partial^2 u}{\partial \theta^2}+ F^{\prime\prime}(u) \left(\frac{\partial u}{\partial \theta}\right)^2 + F(u)\right] \kappa
\\
&&+u^2 \left[F^{\prime\prime}(u) \frac{\partial^2 u}{\partial \theta^2}+ F^{\prime\prime\prime}(u) \left(\frac{\partial u}{\partial \theta}\right)^2 + F^{\prime}(u)\right] \kappa
\\
&&- \frac{u}{\pi} \left(\int_0^{2\pi} F(u) d\theta\right) \kappa - \frac{u^2}{2\pi}\int_0^{2\pi} F^{\prime}(u)\kappa d\theta.
\end{eqnarray*}
The smooth function $\kappa_0(\theta)$ is positive everywhere, so the linearization of (\ref{eq:2.2.201811}) at $u_0$ with small $t>0$ is uniformly parabolic.
It follows from the linearization method \cite{Chou-Zhu-2001, Gao-Zhang-2019} that the Cauchy problem (\ref{eq:2.2.201811})
with initial value $\kappa_0$ has a unique smooth solution $\kappa (\theta, t)$ which satisfies the closing condition (\ref{eq:2.3.201811}) on some time interval.
Thus the flow (\ref{eq:2.1.201811}) and also the flow (\ref{eq:1.1.201811}) exist on a short time interval.

\begin{lemma}\label{lem:2.1.201811}
There is a unique family of smooth convex curves evolving under the flow (\ref{eq:1.1.201811}) on a time interval
$[0, t)$, where $t$ is positive.
\end{lemma}

Using the maximum principle, one can show that $\kappa$ is always positive under the flow.
So the flow (\ref{eq:2.1.201811}) preserves the convexity of the evolving curve.

\begin{lemma}\label{lem:2.2.201811}
Under the flow (\ref{eq:2.1.201811}), the evolving curve is convex.
\end{lemma}
\begin{proof}
Suppose that the flow (\ref{eq:2.1.201811}) exits on time interval $[0, \omega)$.
By the continuity of the evolving curve, it is strictly convex in a small time interval.
Suppose there is a positive $t_0 <\omega$ such that $\min_\theta \kappa(\theta) (t_0)=0$ but $X(\cdot, t)$ is strictly convex when $t\in [0, t_0)$.

By the continuity of $\kappa$ on $[0, 2\pi]\times [0, t_0]$, there exists a positive constant $M$ such that $\max\{\kappa(\theta, t)| \theta \in [0, 2\pi], t\in [0, t_0]\} \leq M$.
Define a function on the interval $[0, \omega)$ as follows
\begin{eqnarray*}
f(t)=\frac{\min\limits_\theta \kappa_0(\theta)}{F(M)\min\limits_\theta \kappa_0(\theta)t+1}.
\end{eqnarray*}
The positive function $f$ satisfies that $f(0)=\min\limits_\theta \kappa_0(\theta)\leq \kappa_0(\theta)$ and $f^\prime (t)=-F(M)f^2(t)$.
Define $h (\theta, t)=\kappa(\theta, t)-f(t)+A\varepsilon t + \varepsilon$ on the domain $[0, 2\pi] \times [0, t_0)$,
where $\varepsilon >0$ and $A = F(M)(M + \min\limits_\theta \kappa_0(\theta))$. By the evolution equation of $\kappa$,
\begin{eqnarray*}
\frac{\partial h}{\partial t} &=& \kappa^2 \left[F^\prime (\kappa) \frac{\partial^2 h}{\partial \theta^2}
            +F^{\prime\prime}(\kappa) \left(\frac{\partial h}{\partial \theta}\right)^2\right]
            \\
            && +\kappa^2 F(\kappa) -\kappa^2 \frac{1}{2\pi}\int_0^{2\pi} F(\kappa) d\theta +F(M)f^2(t) + A\varepsilon.
\end{eqnarray*}

Next, we prove that $h(\theta, t) > \frac{\varepsilon}{2}$ holds for all $(\theta, t)\in [0, 2\pi]\times [0, t_0)$. Suppose there is a smallest $t_1 \in (0, t_0)$ such
that $\min\limits_\theta h(\theta, t_1) = \frac{\varepsilon}{2}$ and $h(\theta, t) > \frac{\varepsilon}{2}$ holds for all $(\theta, t)\in [0, 2\pi]\times [0, t_1)$.
By the continuity of $h$, there exists $\theta_1 \in [0, 2\pi]$ such that $h(\theta_1, t_1) = \min\limits_\theta h(\theta, t_1)$. At the point $(\theta_1, t_1)$, we have
$\frac{\partial^2 h}{\partial \theta^2}(\theta_1, t_1) \geq 0, \frac{\partial h}{\partial \theta}(\theta_1, t_1) =0$ and
$$\frac{\partial h}{\partial t}(\theta_1, t_1) = \lim_{t \rightarrow t_1^-} \frac{h(\theta_1, t) - h(\theta_1, t_1)}{t- t_1} \leq 0.$$
However, at the same point $(\theta_1, t_1)$, the R.H.S. of the evolution equation of $h$ satisfies
\begin{eqnarray*}
&&\kappa^2 F(\kappa) -\kappa^2 \frac{1}{2\pi}\int_0^{2\pi} F(\kappa)  d\theta +F(M)f^2(t_1) + A\varepsilon\\
&&\geq 0 -\kappa^2 F(M) +F(M)f^2(t_1) + A\varepsilon
\\
&&=  -F(M) (\kappa+f(t_1))(h-A\varepsilon t_1 - \varepsilon) + A\varepsilon
\\
&&=  -F(M) (\kappa+f(t_1)) \frac{\varepsilon}{2} + A\varepsilon + F(M) (\kappa+f(t_1))(A\varepsilon t_1 + \varepsilon)
\\
&&\geq  -F(M)(M+ \min_\theta \kappa_0(\theta))\frac{\varepsilon}{2} + A\varepsilon >0.
\end{eqnarray*}
This is contradict to the previous inequality $\frac{\partial h}{\partial t}(\theta_1, t_1) \leq 0$.
Therefore, we have proven the claim of $h$, which implies $\kappa(\theta, t)-f(t)+A\varepsilon t > \varepsilon/2$, if $(\theta, t)\in [0, 2\pi]\times [0, t_0)$.
By the arbitrariness of $\varepsilon$, one obtains that $\kappa (\theta, t) \geq f(t)$ holds on the domain $[0, 2\pi] \times [0, t_0)$.
The continuity of $\kappa$ tells us
$$\kappa (\theta, t_0) \geq f(t_0) = \frac{\min\limits_\theta \kappa_0(\theta)}{F(M)\min\limits_\theta \kappa_0(\theta)t_0+1} >0.$$
So the original assumption that $\kappa$ attains 0 at the time $t_0$ is not correct. The curvature is always positive for $t\in [0, \omega)$.
\end{proof}

The proof of Lemma \ref{lem:2.2.201811} is inspired by the previous work \cite{Jiang-Pan-2008, Mao-Pan-Wang-2013}. Since the evolving curve is always
uniformly convex under the flow (\ref{eq:2.1.201811}), from now on, one can use the tangent angle $\theta$ to parameterize the curve $\widetilde{X}(\cdot, t)$.

\subsection{Uniform bounds of the support function}\label{sec:2.2.201811}
Denote by $p:=-\langle \widetilde{X}, N\rangle$ the support function of the evolving curve. This function
plays an essential role in the proof of the global existence of the flow (\ref{eq:1.1.201811}).

\begin{lemma}\label{lem:2.3.201811}
Under the flow (\ref{eq:2.1.201811}), the length of $\widetilde{X}(\cdot, t)$ is preserved and the
bounded area $A(t)$ is increasing.
\end{lemma}
\begin{proof}
The proof is a direct computation (see also Remark 2.13 in the paper \cite{Lin-Tsai-2012}). Under the flow (\ref{eq:2.1.201811}), the length of the
evolving curve and its bounded area evolve according to next equations respectively.
\begin{eqnarray*}
&&\frac{dL}{dt}=-\int_0^{2\pi} (F(\kappa)-\lambda(t))d\theta
     =-\int_0^{2\pi} \left[F(\kappa)-\frac{1}{2\pi}\int_0^{2\pi} F(\kappa)d\theta \right]d\theta \equiv 0,\\
&&\frac{dA}{dt}=-\int_0^{L} (F(\kappa)-\lambda(t))ds
     =-\int_0^{L} F(\kappa) ds + \frac{L}{2\pi}\int_0^{L} F(\kappa)\kappa ds.
\end{eqnarray*}
Since $F$ is strictly increasing, Andrews' inequality (Lemma I3.3 of the paper \cite{Andrews-1998}) implies that the
area $A(t)$ is increasing.
\end{proof}

Let $I$ be the isoperimetric ratio of a convex curve, i.e., $I=\frac{L^2}{4\pi A}$. For this curve,
Bonnesen's inequality \cite{Osserman-1979} says
\begin{eqnarray*}
rL-A-\pi r^2 \geq 0,  ~~~r\in [r_{in}, r_{out}],
\end{eqnarray*}
where $r_{in}$ and $r_{out}$ are the inradius and outradius
(circumradius) of the convex domain bounded by $X$, respectively.
It follows from Bonnesen's inequality,
\begin{eqnarray*}
\frac{L-\sqrt{L^2-4\pi A}}{2\pi} \leq r_{in} \leq r_{out} \leq \frac{L+\sqrt{L^2-4\pi A}}{2\pi},
\end{eqnarray*}
which implies
\begin{eqnarray*}
\frac{r_{out}}{r_{in}} \leq \frac{L+\sqrt{L^2-4\pi A}}{L-\sqrt{L^2-4\pi A}} =
\left(\sqrt{I}+\sqrt{I-1}\right)^2.
\end{eqnarray*}

Lemma \ref{lem:2.3.201811} tells us that the function $I(t)$ is decreasing under the flow (\ref{eq:1.1.201811}).
In particular, we have
\begin{eqnarray*}
\frac{r_{out}(t)}{r_{in}(t)} \leq \left(\sqrt{I(0)}+\sqrt{I(0)-1}\right)^2, ~~t\in[0, \omega).
\end{eqnarray*}
Noticing that $r_{out}(t) \geq \frac{L}{2\pi}$, we obtain a uniform lower bound of $r_{in}(t)$ as follows
\begin{eqnarray}\label{eq:2.4.201811}
r_{in}(t) \geq \frac{L}{2\pi}\left(\sqrt{I(0)}+\sqrt{I(0)-1}\right)^{-2}>0,  ~~t\in[0, \omega).
\end{eqnarray}
Let $O$ be the center of a maximum inscribed circle of $X_0$. Let $E_0$ be a circle with fixed radius defined by
$$r_0:=\frac{L}{2\pi}\left(\sqrt{I(0)}+\sqrt{I(0)-1}\right)^{-2}$$
and centred at $O$. Shrink $E_0$ via next flow
\begin{equation}\label{eq:2.5.201811}
\left\{\begin{array}{ll}
\frac{\partial Y}{\partial t}= F(\widetilde{\kappa}) N_{in},\\\\
Y(\cdot, 0) = E_0,
\end{array} \right.
\end{equation}
where $\widetilde{\kappa}$ is the curvature of $Y$. By the maximum principle, $Y(\cdot, t)$  is a family of circles.
Its radius $r=r(t)$ satisfies
\begin{eqnarray*}
\frac{dr}{dt}=-F\left(\frac{1}{r}\right).
\end{eqnarray*}
If $r(t)$ is positive then $F\left(\frac{1}{r(t)}\right)>0$. By the evolution equation of $r(t)$, this function is strictly decreasing.
Define
\begin{equation}\label{eq:2.6.201811}
T_1:=r_0\bigg/\left[2F\left(\frac{2}{r_0}\right)\right].
\end{equation}
We will show that, if $t\in [0, T_1)$, then
\begin{eqnarray*}
r(t)>\frac{r_0}{2}.
\end{eqnarray*}
Assume that there is a positive $\widetilde{t} \in (0, T_1)$ such that $r(t)>\frac{r_0}{2}$ holds
for $t\in (0, \widetilde{t})$ and $r(\widetilde{t}) = \frac{r_0}{2}$. Since $F$ is increasing,
for time $t\in (0, \widetilde{t})$, one obtains
\begin{eqnarray*}
F\left(\frac{1}{r(t)}\right) < F\left(\frac{1}{\frac{r_0}{2}}\right) = F\left(\frac{2}{r_0}\right).
\end{eqnarray*}
So, on the time interval $[0, \widetilde{t})$, the radius function satisfies
\begin{eqnarray*}
\frac{dr}{dt} > -F\left(\frac{2}{r_0}\right).
\end{eqnarray*}
Integrating this inequality on time interval $[0, t)$ gives us
\begin{eqnarray}\label{eq:2.7.201811}
r(t) \geq r_0-F\left(\frac{2}{r_0}\right)t.
\end{eqnarray}
Let $t$ tend to $\widetilde{t}$. By the continuity of $r(t)$ and the definitions of $\widetilde{t}$ and $T_1$, we have
\begin{eqnarray*}
r(\widetilde{t}) \geq r_0-F\left(\frac{2}{r_0}\right)\widetilde{t}
> r_0 - F\left(\frac{2}{r_0}\right)T_1
=  \frac{r_0}{2}.
\end{eqnarray*}
This conflicts the existence of $\widetilde{t}$. Therefore, $r(t)\geq \frac{r_0}{2}$
if $t\in [0, T_1]$.

Since $X(\cdot, t)$ contains $Y(\cdot, t)$ for each $t\in \left[0, \min\left\{T_1, \omega\right\}\right)$, the support function has a lower bound
\begin{eqnarray*}
p(\theta, t) \geq r(t) \geq r(T_1)\geq \frac{r_0}{2}
\end{eqnarray*}
in the same time interval. Setting $\delta =\frac{r_0}{4}>0$ and by the definition of the support function, we have, on the domain $[0, 2\pi] \times \big[0, \min\left\{T_1, \omega\right\}\big)$,
next estimate
\begin{eqnarray}\label{eq:2.8.201811}
2\delta \leq p(\theta, t) \leq \frac{L}{2},
\end{eqnarray}
where the constants $\delta$ and $L$ only rely on the initial curve $X_0$.

\subsection{Extending the flow globally}\label{sec:2.3.201811}
 We argue by contradiction. Suppose the flow (\ref{eq:1.1.201811}) exists on the maximal time interval $[0, \omega)$ and $\omega$ is finite.
 Let $n$ be the unique nonnegative integer such that $nT_1 < \omega \leq (n+1)T_1$, where $T_1$ is defined in the equation (\ref{eq:2.6.201811}).
On each time interval $$[0, T_1), [T_1, 2T_1), \cdots, [nT_1, \omega),$$
one can choose an original point $O$ such that the support function $p(\theta, t)$ with respect to $O$ satisfying (\ref{eq:2.8.201811}).

On the above each time interval, let us consider the function
$$\varphi (\theta, t):= \frac{F(\kappa(\theta, t))}{p(\theta, t)-\delta}.$$
By some straightforward calculations, we have
\begin{eqnarray*}
\frac{\partial \varphi}{\partial \theta} = \frac{1}{p-\delta}\frac{\partial F}{\partial \theta}
   -\frac{F}{(p-\delta)^2}\frac{\partial p}{\partial \theta},
\end{eqnarray*}
and
\begin{eqnarray*}
\frac{\partial^2 \varphi}{\partial \theta^2} = \frac{1}{p-\delta}\frac{\partial^2 F}{\partial \theta^2}
   -\frac{2}{(p-\delta)^2}\frac{\partial F}{\partial \theta}\frac{\partial p}{\partial \theta}
   -\frac{F}{(p-\delta)^2}\frac{\partial^2 p}{\partial \theta^2}
   +\frac{2F}{(p-\delta)^3}\left(\frac{\partial p}{\partial \theta}\right)^2.
\end{eqnarray*}
By the evolution equation of $\kappa$, the function $F$ evolves according to
\begin{eqnarray}\label{eq:2.9.201811}
\frac{\partial F}{\partial t} = F^\prime \kappa^2 \left[\frac{\partial^2 F}{\partial \theta^2} +F-\lambda(t) \right].
\end{eqnarray}
It follows from the definition of the support function and the flow equation (\ref{eq:2.1.201811}), the support function satisfies
\begin{eqnarray}\label{eq:2.10.201811}
\frac{\partial p}{\partial t} =\lambda(t)-F.
\end{eqnarray}
Direct calculations can give us the evolution equation of $\varphi (\theta, t)$:
\begin{eqnarray*}
\frac{\partial \varphi}{\partial t} &=& \frac{1}{p-\delta} \frac{\partial F}{\partial t} - \frac{F}{(p-\delta)^2} \frac{\partial p}{\partial t}
\\
&=& \frac{F^\prime \kappa^2}{p-\delta} \left(\frac{\partial^2 F}{\partial \theta^2} +F-\lambda(t) \right) - \frac{F}{(p-\delta)^2} (\lambda(t) -F)
\\
&=& F^\prime \kappa^2 \left[\frac{\partial^2 \varphi}{\partial \theta^2}
   +\frac{2}{(p-\delta)^2}\frac{\partial F}{\partial \theta}\frac{\partial p}{\partial \theta}
   +\frac{F}{(p-\delta)^2}\frac{\partial^2 p}{\partial \theta^2}
   -\frac{2F}{(p-\delta)^3}\left(\frac{\partial p}{\partial \theta}\right)^2\right]
   \\
   && +\frac{F^\prime \kappa^2}{p-\delta}(F-\lambda(t)) - \frac{F}{(p-\delta)^2} (\lambda(t) -F)
\\
&=& F^\prime \kappa^2 \frac{\partial^2 \varphi}{\partial \theta^2}
    + \frac{2F^\prime \kappa^2}{(p-\delta)^2}\frac{\partial p}{\partial \theta} \left[\frac{\partial \varphi}{\partial \theta} +\frac{F}{(p-\delta)^2}\frac{\partial p}{\partial \theta}\right]
    + F^\prime \kappa^2 \frac{F}{(p-\delta)^2}\frac{\partial^2 p}{\partial \theta^2}
    \\
    && -F^\prime \kappa^2\frac{2F}{(p-\delta)^3}\left(\frac{\partial p}{\partial \theta}\right)^2
      +\frac{F^\prime \kappa^2}{p-\delta}(F-\lambda(t)) - \frac{F}{(p-\delta)^2} (\lambda(t) -F)
\\
&=& F^\prime \kappa^2 \frac{\partial^2 \varphi}{\partial \theta^2}
    + \frac{2F^\prime \kappa^2}{(p-\delta)^2}\frac{\partial p}{\partial \theta}\frac{\partial \varphi}{\partial \theta}
    + F^\prime \kappa^2 \frac{F}{(p-\delta)^2}\left(\frac{\partial^2 p}{\partial \theta^2} +p -\delta\right)
    \\
    && -\frac{F^\prime \kappa^2}{p-\delta} \lambda(t) -\frac{\lambda(t)F}{(p-\delta)^2} + \left(\frac{F}{p-\delta}\right)^2
\\
&=& F^\prime \kappa^2 \frac{\partial^2 \varphi}{\partial \theta^2}
      +\frac{2F^\prime \kappa^2}{p-\delta}\frac{\partial p}{\partial \theta}\frac{\partial \varphi}{\partial \theta}
      +\frac{F^\prime \kappa \varphi}{p-\delta}-\frac{\delta F^\prime \kappa^2 \varphi}{p-\delta}
      -\frac{\lambda F^\prime \kappa^2}{p-\delta}-\frac{\lambda \varphi}{p-\delta}+\varphi^2.
\end{eqnarray*}
In the above calculation, the identity $\frac{\partial^2 p}{\partial \theta^2}+p=\frac{1}{\kappa}$ is used.

\begin{lemma}\label{lem:2.4.201811}
There is a positive constant $M$ independent of time $t$, $\theta$ and the finite number $\omega$ such that
\begin{eqnarray}\label{eq:2.11.201811}
\kappa(\theta, t) \leq M
\end{eqnarray}
for all $(\theta, t)\in [0, 2\pi]\times [0, \omega)$.
\end{lemma}
\begin{proof}
Consider the time interval $\big[0, \min\{T_1, \omega\}\big)$. We will prove (\ref{eq:2.11.201811}) by bounding the function $\varphi(\theta, t)$.
In the evolution equation of $\varphi$, the term $-\frac{\delta F^\prime \kappa^2 \varphi}{p-\delta}$ will be used to bound positive terms
$\frac{F^\prime \kappa \varphi}{p-\delta}$ and $\varphi^2$.
Rewrite the evolution equation of $\varphi$ as
\begin{eqnarray}\label{eq:2.12.201811}
\frac{\partial \varphi}{\partial t} &=& F^\prime \kappa^2 \frac{\partial^2 \varphi}{\partial \theta^2}
      +\frac{2F^\prime \kappa^2}{p-\delta}\frac{\partial p}{\partial \theta}\frac{\partial \varphi}{\partial \theta}
      -\frac{\lambda F^\prime \kappa^2}{p-\delta}-\frac{\lambda \varphi}{p-\delta} \nonumber\\
&&    +\frac{F^\prime \kappa}{p-\delta}\left(1-\frac{\delta}{2}\kappa\right)\varphi
      -\left(\frac{\delta F^\prime \kappa^2}{2F}-1\right) \varphi^2.
\end{eqnarray}
The condition (iii) of $F$ says $\lim\limits_{u\rightarrow +\infty} \frac{F^\prime (u)\cdot u^2}{F(u)}=+\infty$, so there exists a constant $U_0>0$ such that
$\frac{F^\prime (u)\cdot u^2}{F(u)}>\frac{2}{\delta}$ when $u > U_0$.

Suppose that $\varphi (\theta, t)$ attains its maximum at a point $(\theta_*, t)$ with respect to $\theta$.
If
\begin{eqnarray}\label{eq:2.13.201811}
\varphi(\theta_*, t)> \frac{1}{\delta} F\left(\max\left\{\frac{2}{\delta}, U_0\right\}\right)
\end{eqnarray}
then (\ref{eq:2.8.201811}) tells us that
\begin{eqnarray*}
F(\kappa(\theta_*, t))= (p-\delta)\varphi(\theta_*, t) > \frac{p-\delta}{\delta}F\left(\max\left\{\frac{2}{\delta}, U_0\right\}\right)
   >F\left(\max\left\{\frac{2}{\delta}, U_0\right\}\right).
\end{eqnarray*}
Noticing that $F(u)$ is increasing with respect to $u$, one obtains
$$\kappa(\theta_*, t)>\max\left\{\frac{2}{\delta}, U_0\right\}.$$
So we have $1-\frac{\delta}{2}\kappa(\theta_*, t) <0$. By the existence of $U_0$ and the above lower bound of the curvature $\kappa (\theta_*, t) >U_0$, one obtains
\begin{eqnarray*}
\frac{\delta F^\prime (\kappa(\theta_*, t))\cdot \kappa(\theta_*, t)^2}{2F(\kappa(\theta_*, t))}-1 >0.
\end{eqnarray*}
Therefore, (\ref{eq:2.13.201811}) implies
\begin{eqnarray*}
\frac{\partial \varphi}{\partial t}(\theta_*, t) \leq 0.
\end{eqnarray*}
By the above discussion, the function $\varphi_{\max}(t): = \max\{\varphi (\theta, t)|\theta \in [0, 2\pi]\}$
is decreasing if (\ref{eq:2.13.201811}) holds.
So the maximum principle implies that
\begin{eqnarray}\label{eq:2.14.201811}
\varphi(\theta, t) \leq \max
\left\{\max_\theta \varphi(\theta, 0), ~\frac{1}{\delta} F\left(\max\left\{\frac{2}{\delta}, U_0\right\}\right)\right\}:=C_0,
\ \ t\in [0, \min\{T_1, \omega\}),
\end{eqnarray}
where the number $T_1$ only depends on the initial curve and the function $F$.

If $\omega>T_1$ then by similar discussion, one can prove
\begin{eqnarray*}
\varphi(\theta, t) \leq \max
\left\{\max_\theta \varphi(\theta, T_1), ~\frac{1}{\delta} F\left(\max\left\{\frac{2}{\delta}, U_0\right\}\right)\right\}
\end{eqnarray*}
for $t\in [T_1, \min\{2T_1, \omega\})$. Since the estimate (\ref{eq:2.14.201811}) gives us an upper bound of $\max_\theta \varphi(\theta, T_1)$, one also has
$\varphi(\theta, t) \leq C_0$, where $t\in [T_1, \min\{2T_1, \omega\})$.

By the mathematical induction, there is a positive constant $C_0$ (the R.H.S. of (\ref{eq:2.14.201811})) such that
$\varphi(\theta, t) \leq C_0$ for all $t\in [0, \omega)$.  It follows from (\ref{eq:2.8.201811}) that
\begin{eqnarray}\label{eq:2.15.201811}
F(\kappa) \leq \left(\frac{L}{2}-\delta\right)C_0, \ \ \ \ \ \ t\in [0, \omega).
\end{eqnarray}
Therefore, the condition (ii) of the function $F$ implies that
the curvature $\kappa$ has an upper bound $M$ independent of time and the number $\omega$.
\end{proof}

Once the curvature is bounded uniformly from above, the flow can be extended on time interval $[0, +\infty)$.

\begin{theorem}\label{thm:2.5.201811}
The flow (\ref{eq:2.1.201811}) exists on time interval $[0, +\infty)$.
\end{theorem}
\begin{proof}
Suppose the flow (\ref{eq:2.1.201811}) exists on the maximal time interval $[0, \omega)$, where $\omega$ is finite.
For the fixed time $t\in (0, \omega)$, the assumption implies that $\kappa$ is uniformly bounded on the time interval $[0, t]$. So it follows from Lemma \ref{lem:2.2.201811}
that $\kappa$ is positive on $[0, 2\pi] \times [0, t]$. The number $t$ can be chosen arbitrarily close to $\omega$, hence the evolving curve is proved to be convex
on the time interval $[0, \omega)$. By Lemma \ref{lem:2.4.201811}, the curvature function has a uniformly upper bounded $M$ on the time interval $[0, \omega)$,
where the constant $M> 0$ only relies on the initial curve $X_0$.

Applying Lemma \ref{lem:2.2.201811} again, there is a constant $m=m(\omega, M)>0$ such that
\begin{eqnarray}\label{eq:2.16.201811}
\kappa (\theta, t) \geq m, \ \ \ \ \ \ (\theta, t) \in [0, 2\pi] \times [0, \omega).
\end{eqnarray}
Therefore, the evolution equation of $\kappa$ is uniformly parabolic on time interval $[0, \omega)$.
Since $\kappa$ has uniformly bounds (see (\ref{eq:2.11.201811}) and (\ref{eq:2.16.201811})) and
$F(\cdot)$ is smooth and increasing on $(0, +\infty)$, the functions $F(\kappa)$ and $\lambda (t)$ can be bounded on
time interval $[0, \omega)$ as follows
$$ F(m) \leq F(\kappa)\leq F(M), ~~ F(m)\leq \lambda (t)\leq F(M).$$
By Lemmas \ref{lem:2.6.201811} and \ref{lem:2.7.201811} in the next subsection, each derivative $|\frac{\partial^k F(\kappa)}{\partial \theta^k}|$
has a uniform bound on the bounded domain $[0, 2\pi] \times [0, \omega)$. Since $F(\kappa)$ also has uniformly bounds, there exists a convergent subsequence
$F(\kappa(\theta, t_i)$ tending to a limit
$$F(\kappa(\theta, \omega)):=\lim_{t_i \uparrow \omega} F(\kappa(\theta, t_i)).$$
According to the evolution equation of $F$, the time derivative $\frac{\partial F}{\partial t}$ also has uniform bounds. By the method in the paper \cite{Gao-Pan-Tsai-2020-2},
one has the full-time uniform convergence of $F(\kappa(\theta, t))$ as $t \rightarrow \omega$.
The solution to the flow (\ref{eq:1.1.201811}) is unique, so integrating the flow equation (\ref{eq:1.1.201811}) gives us a smooth convex curve $X(\cdot, \omega)$.

Let the convex curve $X(\cdot, \omega)$ evolve according to the flow (\ref{eq:1.1.201811}).
By the short time existence, one can extend this flow on a larger time interval $[0, \omega+\varepsilon)$ where $\varepsilon>0$. This contradicts to the assumption that $\omega$
is maximal. So the flow (\ref{eq:2.1.201811}) exists globally.
\end{proof}

\subsection{Higher regularity of the flow}\label{sec:2.4.201811}
For the higher regularity of the flow, one needs to bound all derivatives of $\kappa$ on any finite time interval $[0, \omega)$, where $\omega$ is a positive number.
The proof relies on both integral estimates (see for example Gage-Hamiton \cite{Gage-Hamilton}, Jiang-Pan \cite{Jiang-Pan-2008}) and a simpler method by Lin and Tsai \cite{Lin-Tsai-2008}.

On the time interval $[0, \omega)$, the estimates (\ref{eq:2.11.201811}), (\ref{eq:2.15.201811}) and (\ref{eq:2.16.201811}) imply that there are positive constants $c_1=c_1(\omega), C_i=C_i(\omega)$ such that
\begin{eqnarray}\label{eq:2.17.201811}
\left|\frac{d^i F}{du^i}(\kappa)\right|\leq C_i, i=1, 2, \cdots, \ \ \frac{d F}{du} (\kappa)\geq c_1.
\end{eqnarray}

\begin{lemma}\label{lem:2.6.201811}
On the time interval $[0, \omega)$, the integral $\int_0^{2\pi} \left(\frac{\partial F}{\partial \theta}\right)^6 d\theta$
has at most polynomial growth.
\end{lemma}
\begin{proof}
Using the evolution equation of $F$ (see \eqref{eq:2.9.201811}), one can compute that
\begin{eqnarray*}
\frac{d}{dt}\frac{1}{6}\int_0^{2\pi} \left(\frac{\partial F}{\partial \theta}\right)^6 d\theta
&=&\int_0^{2\pi} \left(\frac{\partial F}{\partial \theta}\right)^5 \frac{\partial^2 F}{\partial \theta \partial t} d\theta
   =-5\int_0^{2\pi} \left(\frac{\partial F}{\partial \theta}\right)^4
   \frac{\partial^2 F}{\partial \theta^2} \frac{\partial F}{\partial t} d\theta
\\
&=& -5\int_0^{2\pi}\left(\frac{\partial F}{\partial \theta}\right)^4
       \frac{\partial^2 F}{\partial \theta^2} F^\prime \kappa^2
     \left[\frac{\partial^2 F}{\partial \theta^2} +F-\lambda(t) \right] d\theta
\\
&=& -5\int_0^{2\pi}F^\prime \kappa^2 \left(\frac{\partial F}{\partial \theta}\right)^4
       \left(\frac{\partial^2 F}{\partial \theta^2}\right)^2 d\theta
   \\
   &&-5\int_0^{2\pi}F^\prime \kappa^2 (F-\lambda(t)) \left(\frac{\partial F}{\partial \theta}\right)^4
       \frac{\partial^2 F}{\partial \theta^2} d\theta.
\end{eqnarray*}
Substituting this inequality
\begin{eqnarray*}
(F-\lambda(t))\frac{\partial^2 F}{\partial \theta^2}
\leq \frac{1}{2}\left[\left(\frac{\partial^2 F}{\partial \theta^2}\right)^2 +(F-\lambda(t))^2 \right]
\end{eqnarray*}
into the evolution equation of
$\frac{1}{6}\int_0^{2\pi} \left(\frac{\partial F}{\partial \theta}\right)^6 d\theta$, one gets
\begin{eqnarray*}
\frac{d}{dt}\frac{1}{6}\int_0^{2\pi} \left(\frac{\partial F}{\partial \theta}\right)^6 d\theta
&\leq & -\frac{5}{2}\int_0^{2\pi}F^\prime \kappa^2 \left(\frac{\partial F}{\partial \theta}\right)^4
       \left(\frac{\partial^2 F}{\partial \theta^2}\right)^2 d\theta
   \\
   &&+\frac{5}{2}\int_0^{2\pi}F^\prime \kappa^2 (F-\lambda(t))^2 \left(\frac{\partial F}{\partial \theta}\right)^4d\theta.
\end{eqnarray*}
Since both $F(\kappa)$  and  $\lambda(t)$ are positive and $\lambda(t)$ is the average of $F(\kappa (\theta, t))$ with respect to $\theta$,
we have $|F-\lambda| \leq \max \{F, \lambda\} \leq F(M)$.
By (\ref{eq:2.11.201811}) and (\ref{eq:2.17.201811}), one obtains
\begin{eqnarray*}
\frac{d}{dt}\frac{1}{6}\int_0^{2\pi} \left(\frac{\partial F}{\partial \theta}\right)^6 d\theta
&\leq & \frac{5}{2}C_1 M^2 F(M)^2 \int_0^{2\pi} \left(\frac{\partial F}{\partial \theta}\right)^4 d\theta
\\
&\leq & \frac{5}{2}(2\pi)^{1/3}C_1 M^2 F(M)^2
        \left[\int_0^{2\pi} \left(\frac{\partial F}{\partial \theta}\right)^6 d\theta\right]^{2/3}.
\end{eqnarray*}
Integrating this inequality with respect to the time $t$ and using H\"{o}lder's inequality, one gets next estimate
\begin{eqnarray*}
\int_0^{2\pi} \left(\frac{\partial F}{\partial \theta}\right)^6 d\theta
\leq  \left[\left(\int_0^{2\pi} \left(\frac{\partial F_0}{\partial \theta}\right)^6 d\theta\right)^{\frac{1}{3}} +
45 (2\pi)^{1/3}C_1 M^2 F(M)^2 t\right]^{3}.
\end{eqnarray*}
So, $\int_0^{2\pi} \left(\frac{\partial F}{\partial \theta}\right)^6 d\theta$ is bounded by a cubic polynomial for $t\in [0, \omega)$.
\end{proof}

By H\"{o}lder's inequality, $\int_0^{2\pi} \left(\frac{\partial F}{\partial \theta}\right)^2 d\theta
\leq (2\pi)^{2/3}\left[\int_0^{2\pi} \left(\frac{\partial F}{\partial \theta}\right)^6 d\theta\right]^{1/3}$.
Lemma \ref{lem:2.6.201811} tells us that $\int_0^{2\pi} \left(\frac{\partial F}{\partial \theta}\right)^2 d\theta$
has at most linear growth.

\begin{lemma}\label{lem:2.7.201811}
On the time interval $[0, \omega)$, the integral $\int_0^{2\pi} \left(\frac{\partial^2 F}{\partial \theta^2}\right)^2 d\theta$
has at most polynomial growth.
\end{lemma}
\begin{proof}
We first compute by the evolution equation of $F(\kappa)$,
\begin{eqnarray*}
\frac{d}{dt}\frac{1}{2}\int_0^{2\pi} \left(\frac{\partial^2 F}{\partial \theta^2}\right)^2 d\theta
&=&\int_0^{2\pi}\frac{\partial^3 F}{\partial \theta^2 \partial t} \frac{\partial^2 F}{\partial \theta^2} d\theta
=-\int_0^{2\pi}\frac{\partial^2 F}{\partial \theta \partial t} \frac{\partial^3 F}{\partial \theta^3} d\theta
\\
&=& -\int_0^{2\pi}F^\prime \kappa^2 \left(\frac{\partial^3 F}{\partial \theta^3}\right)^2 d\theta
    -\int_0^{2\pi}F^{\prime\prime} \frac{\partial \kappa}{\partial \theta} \kappa^2
         \frac{\partial^2 F}{\partial \theta^2}\frac{\partial^3 F}{\partial \theta^3} d\theta
    \\
    &&-2\int_0^{2\pi}F^\prime \kappa \frac{\partial \kappa}{\partial \theta}
       \frac{\partial^2 F}{\partial \theta^2}\frac{\partial^3 F}{\partial \theta^3} d\theta
       -\int_0^{2\pi}F^{\prime\prime} \frac{\partial \kappa}{\partial \theta} \kappa^2 (F-\lambda(t))
       \frac{\partial^3 F}{\partial \theta^3} d\theta
     \\
    &&-2 \int_0^{2\pi}F^\prime \kappa \frac{\partial \kappa}{\partial \theta}(F-\lambda(t))
       \frac{\partial^3 F}{\partial \theta^3} d\theta
      -\int_0^{2\pi}F^\prime \kappa^2  \frac{\partial F}{\partial \theta}
        \frac{\partial^3 F}{\partial \theta^3} d\theta
\\
&:=& -\int_0^{2\pi}F^\prime \kappa^2 \left(\frac{\partial^3 F}{\partial \theta^3}\right)^2 d\theta
    +I+II+III+IV+V.
\end{eqnarray*}
Using the Cauchy-Schwarz inequality, (\ref{eq:2.11.201811}) and (\ref{eq:2.17.201811}), one can estimate
\begin{eqnarray*}
I &=& -\int_0^{2\pi}F^{\prime\prime} \frac{\partial \kappa}{\partial \theta} \kappa^2
        \frac{\partial^2 F}{\partial \theta^2}\frac{\partial^3 F}{\partial \theta^3} d\theta
 =  -\int_0^{2\pi}\frac{F^{\prime\prime}}{F^\prime}\kappa^2 \frac{\partial F}{\partial \theta}
        \frac{\partial^2 F}{\partial \theta^2}\frac{\partial^3 F}{\partial \theta^3} d\theta
\\
&\leq & \varepsilon \int_0^{2\pi} \left(\frac{\partial^3 F}{\partial \theta^3}\right)^2 d\theta
      +\frac{1}{4\varepsilon} \int_0^{2\pi} \left(\frac{F^{\prime\prime}}{F^\prime}\right)^2\kappa^4 \left(\frac{\partial F}{\partial \theta}\right)^2
       \left(\frac{\partial^2 F}{\partial \theta^2}\right)^2 d\theta,
\\
&\leq & \varepsilon \int_0^{2\pi} \left(\frac{\partial^3 F}{\partial \theta^3}\right)^2 d\theta
      +\frac{C_2^2 M^4}{4\varepsilon c_1^2} \int_0^{2\pi} \left(\frac{\partial F}{\partial \theta}\right)^2
       \left(\frac{\partial^2 F}{\partial \theta^2}\right)^2 d\theta,
\\
II &=& -2\int_0^{2\pi}F^\prime \kappa \frac{\partial \kappa}{\partial \theta}
       \frac{\partial^2 F}{\partial \theta^2}\frac{\partial^3 F}{\partial \theta^3} d\theta
\\
&\leq & \varepsilon \int_0^{2\pi} \left(\frac{\partial^3 F}{\partial \theta^3}\right)^2 d\theta
      +\frac{C_1^2M^2}{\varepsilon} \int_0^{2\pi} \left(\frac{\partial F}{\partial \theta}\right)^2
       \left(\frac{\partial^2 F}{\partial \theta^2}\right)^2 d\theta,
\\
III &=& -\int_0^{2\pi}F^{\prime\prime} \frac{\partial \kappa}{\partial \theta} \kappa^2 (F-\lambda(t))
       \frac{\partial^3 F}{\partial \theta^3} d\theta
\\
&=& -\int_0^{2\pi}\frac{F^{\prime\prime}}{F^{\prime}} \frac{\partial F}{\partial \theta} \kappa^2 (F-\lambda(t))
       \frac{\partial^3 F}{\partial \theta^3} d\theta
\\
&\leq & \varepsilon \int_0^{2\pi} \left(\frac{\partial^3 F}{\partial \theta^3}\right)^2 d\theta
       + \frac{C_2^2 M^4 (F(M))^2}{4\varepsilon c_1^2} \int_0^{2\pi} \left(\frac{\partial F}{\partial \theta}\right)^2 d\theta,
\\
IV &=& -2 \int_0^{2\pi}F^\prime \kappa \frac{\partial \kappa}{\partial \theta}(F-\lambda(t))
       \frac{\partial^3 F}{\partial \theta^3} d\theta
\\
&=& -2 \int_0^{2\pi} \kappa \frac{\partial F}{\partial \theta}(F-\lambda(t))
       \frac{\partial^3 F}{\partial \theta^3} d\theta
\\
&\leq & \varepsilon \int_0^{2\pi} \left(\frac{\partial^3 F}{\partial \theta^3}\right)^2 d\theta
        +\frac{M^2 (F(M))^2}{\varepsilon} \int_0^{2\pi} \left(\frac{\partial F}{\partial \theta}\right)^2 d\theta,
\\
V &=& -\int_0^{2\pi}F^\prime \kappa^2  \frac{\partial F}{\partial \theta}
        \frac{\partial^3 F}{\partial \theta^3} d\theta
\\
&\leq & \varepsilon \int_0^{2\pi} \left(\frac{\partial^3 F}{\partial \theta^3}\right)^2 d\theta
        +\frac{C_1^2 M^4}{4\varepsilon} \int_0^{2\pi} \left(\frac{\partial F}{\partial \theta}\right)^2 d\theta,
\end{eqnarray*}
where we have also used $|F-\lambda| \leq F(M)$. Therefore,
\begin{eqnarray*}
\frac{d}{dt}\frac{1}{2}\int_0^{2\pi} \left(\frac{\partial^2 F}{\partial \theta^2}\right)^2 d\theta
& \leq & \int_0^{2\pi}(5\varepsilon - F^\prime \kappa^2) \left(\frac{\partial^3 F}{\partial \theta^3}\right)^2 d\theta
\\
&& +\left(\frac{C_2^2 M^4}{4\varepsilon c_1^2}+ \frac{C_1^2M^2}{\varepsilon}\right)
      \int_0^{2\pi}\left(\frac{\partial F}{\partial \theta}\right)^2
       \left(\frac{\partial^2 F}{\partial \theta^2}\right)^2 d\theta
\\
&& + \left(\frac{C_2^2 M^4 (F(M))^2}{4\varepsilon c_1^2} + \frac{M^2 (F(M))^2}{\varepsilon}
           + \frac{C_1^2 M^4}{4\varepsilon}\right) \int_0^{2\pi} \left(\frac{\partial F}{\partial \theta}\right)^2 d\theta.
\end{eqnarray*}
Since
\begin{eqnarray*}
\int_0^{2\pi}\left(\frac{\partial F}{\partial \theta}\right)^2
            \left(\frac{\partial^2 F}{\partial \theta^2}\right)^2 d\theta
&= & -\frac{1}{3} \int_0^{2\pi}\left(\frac{\partial F}{\partial \theta}\right)^3
            \frac{\partial^3 F}{\partial \theta^3} d\theta
\\
&\leq & \frac{\varepsilon^2}{3} \int_0^{2\pi}\left(\frac{\partial^3 F}{\partial \theta^3}\right)^2 d\theta
      +\frac{1}{12\varepsilon^2} \int_0^{2\pi} \left(\frac{\partial F}{\partial \theta}\right)^6 d\theta,
\end{eqnarray*}
one can finally obtain that
\begin{eqnarray*}
\frac{d}{dt}\frac{1}{2}\int_0^{2\pi} \left(\frac{\partial^2 F}{\partial \theta^2}\right)^2 d\theta
& \leq & \int_0^{2\pi} \left[5\varepsilon - F^\prime \kappa^2+\frac{\varepsilon}{3}\left(\frac{C_2^2 M^4}{4 c_1^2}+ C_1^2M^2\right)\right]
        \left(\frac{\partial^3 F}{\partial \theta^3}\right)^2 d\theta
\\
&& +\frac{1}{12\varepsilon^2}\left(\frac{C_2^2 M^4}{4\varepsilon c_1^2}+ \frac{C_1^2M^2}{\varepsilon}\right)
      \int_0^{2\pi}\left(\frac{\partial F}{\partial \theta}\right)^6 d\theta
\\
&& + \left(\frac{C_2^2 M^4 (F(M))^2}{4\varepsilon c_1^2} + \frac{M^2 (F(M))^2}{\varepsilon}
           + \frac{C_1^2 M^4}{4\varepsilon}\right) \int_0^{2\pi} \left(\frac{\partial F}{\partial \theta}\right)^2 d\theta.
\end{eqnarray*}

By (\ref{eq:2.16.201811}) and (\ref{eq:2.17.201811}), one can choose positive number $\varepsilon=\varepsilon(c_1, m, C_1, C_2, M)$ small enough such that
\begin{eqnarray*}
&&5\varepsilon - F^\prime \kappa^2+\frac{\varepsilon}{3}\left(\frac{C_2^2 M^4}{4 c_1^2}+ C_1^2M^2\right)
\\
&& \leq 5\varepsilon - c_1 m^2+\frac{\varepsilon}{3}\left(\frac{C_2^2 M^4}{4 c_1^2}+ C_1^2M^2\right) <0.
\end{eqnarray*}
By Lemma \ref{lem:2.6.201811}, both $\int_0^{2\pi} \left(\frac{\partial F}{\partial \theta}\right)^6 d\theta$
and $\int_0^{2\pi} \left(\frac{\partial F}{\partial \theta}\right)^2 d\theta$ have at most
polynomial growth if $t\in [0, \omega)$. Integrating the above differential inequality can show that
$\int_0^{2\pi} \left(\frac{\partial^2 F}{\partial \theta^2}\right)^2 d\theta$ is bounded by a polynomial with degree at most 4.
\end{proof}

By Sobolev's inequality, the function $\frac{\partial F(\kappa (\theta, t))}{\partial \theta}$ is bounded on the domain $[0, 2\pi] \times [0, \omega)$.
There is a constant $\widetilde{M}_1 = \widetilde{M}_1(X_0, \omega)$ such that
\begin{eqnarray}\label{eq:2.18.201811}
\left|\frac{\partial F(\kappa (\theta, t))}{\partial \theta} \right| \leq \widetilde{M}_1, \ \ \ \ \ \ (\theta, t) \in [0, 2\pi] \times [0, \omega).
\end{eqnarray}

\begin{lemma}\label{lem:2.8.201811}
There exists a positive constant $\widetilde{M}_i$, relying on the initial curve $X_0$ and the number $\omega$, such that
\begin{eqnarray}\label{eq:2.19.201811}
\left| \frac{\partial^i }{\partial \theta^i}F(\kappa(\theta, t))\right| \leq \widetilde{M}_i, ~~(\theta, t)\in [0, 2\pi] \times [0, \omega), ~~i=2, 3, \cdots.
\end{eqnarray}
\end{lemma}
\begin{proof}
Let $\alpha$ be a constant to be determined and set
$g = \frac{\partial^2 F}{\partial \theta^2} +\alpha \left(\frac{\partial F}{\partial \theta} \right)^2$.
One may compute from \eqref{eq:2.9.201811} and the fact $\frac{\partial^2 \kappa}{\partial \theta^2}
= \frac{1}{F^{\prime}} \left(\frac{\partial^2 F}{\partial \theta^2} - F^{\prime\prime}(\kappa) \left( \frac{\partial \kappa}{\partial \theta}\right)^2 \right)$ to obtain
\begin{eqnarray}
\frac{\partial g}{\partial t} &=& F^\prime \kappa^2 \frac{\partial^2 g}{\partial \theta^2}
    +\left(4\kappa \frac{\partial F}{\partial \theta}+ 2F^{\prime\prime} \kappa^2 \frac{\partial \kappa}{\partial \theta}
    -2\alpha F^{\prime} \kappa^2 \frac{\partial F}{\partial \theta} \right) \frac{\partial g}{\partial \theta}
    \nonumber\\
    &&+\left(-2\alpha F^\prime \kappa^2 + \frac{F^{\prime\prime}}{F^\prime}\kappa^2 +2\kappa\right)g^2
    + \text{linear terms of} ~g. \label{eq:2.20.201811}
\end{eqnarray}
It follows from (\ref{eq:2.11.201811}) and (\ref{eq:2.16.201811}) that both $\kappa$ and $F^{\prime}(\kappa)$ have positive lower bounds,
$$\kappa \geq m(\omega, M)>0, ~~F^{\prime}(\kappa) \geq F^{\prime}(m)>0, ~~(\theta, t) \in [0, 2\pi] \times [0, \omega).$$
Since the linear terms of $g$ in the equation \eqref{eq:2.20.201811} have bounded coefficients,
one can choose $\alpha>0$ large enough so that the coefficient of $g^2$ in the above equation has negative upper bound
$$-2\alpha F^\prime \kappa^2 + \frac{F^{\prime\prime}}{F^\prime}\kappa^2 +2\kappa \leq c(\omega)<0.$$
By the maximum principle, the function $g_{\max} (t)$ is uniformly bounded on the time interval $[0, \omega)$.
Therefore, $g$ has an uniformly upper bound on time interval $[0, \omega)$.
Similarly, one can choose $\alpha <0$ with sufficiently large $|\alpha|$
so that the coefficient of $g^2$ is uniform positive and $g$ has a lower bound on time interval $[0, \omega)$. This shows that
$|\frac{\partial^2 F}{\partial \theta^2}|$ is uniformly bounded on $[0, 2\pi] \times [0, \omega)$.

Since every higher derivative $\frac{\partial^k F}{\partial \theta^k}(k=3, 4, \cdots)$
satisfies a linear parabolic equation with uniformly bounded coefficients, its absolute value  $\left| \frac{\partial^k F}{\partial \theta^k}\right|$
is uniformly bounded on $[0, 2\pi] \times [0, \omega)$.
\end{proof}

By the uniform positive lower and upper bounds of $\kappa$ and the relations between
$\frac{\partial^k F}{\partial \theta^k}$ and $\frac{\partial^k \kappa}{\partial \theta^k}$, each derivative of $\kappa$ can be
uniformly bounded on the domain $[0, 2\pi] \times [0, \omega)$. This means that the evolving curve $X(\cdot, t)$ is smooth for every
$t\in [0, \omega)$.

Until now we only prove that $\kappa$ is always positive. Whether the evolution equation of $\kappa$ is
degenerate or not as $t\rightarrow +\infty$ is still unknown. In fact, the existence of a uniformly lower
positive bound of the curvature plays an essential role in the study of asymptotic behavior of the evolving curve.
We leave the proof to the next section.

\section{Convergence of the curvature}\label{sec:3.201811}
The main task of this section is to prove that the curvature $\kappa(\theta, t)$ converges to a constant $\frac{2\pi}{L}$
as time tends to infinity.

\begin{remark}\label{rmk:3.1.201811} In the condition (i), we assume that $F^\prime (u)u$ tends to 0 as $u \rightarrow 0$.
In fact, once the limit $\lim\limits_{u\rightarrow 0^+} F^\prime (u)\cdot u$ exists, other conditions of the function $F$ imply that this limit is 0, i.e.,
\begin{eqnarray}\label{eq:3.1.201811}
\lim_{u\rightarrow 0^+} F^\prime (u)\cdot u=0.
\end{eqnarray}
\end{remark}
One may explain (\ref{eq:3.1.201811}) as follows. Since $F^\prime (u)u >0$ holds for $u>0$, one obtains the lower limit
\begin{eqnarray*}
\liminf_{u\rightarrow 0^+} F^\prime (u)\cdot u \geq 0.
\end{eqnarray*}
If one can prove this lower limit equals to $0$, then by the existence of $\lim\limits_{u\rightarrow 0^+} F^\prime (u)\cdot u$, we complete the proof.
Suppose that the lower limit is positive.
There exist constants $\delta>0$ and $C_0>0$ such that $F^\prime (u)\cdot u >C_0$ for $u\in (0, \delta)$.
So one obtains that $F^\prime (u)> C_0/u$ on the interval $(0, \delta)$. Integrating this inequality yields
\begin{eqnarray*}
F(\delta) > F(\varepsilon) + C_0(\ln \delta -\ln \varepsilon) > C_0(\ln \delta -\ln \varepsilon),
\end{eqnarray*}
where $\varepsilon\in(0, \delta)$ can be chosen as an arbitrary number. Letting $\varepsilon$ tend to 0, we
find that $F(\delta)$ is unbounded, which is a contradiction.

Now let us go back to the flow. Using the evolution equation of $F$ (see (\ref{eq:2.8.201811})), one can compute that
\begin{eqnarray*}
&&\frac{d}{dt}\frac{1}{2}\int_0^{2\pi} \left(\frac{\partial F}{\partial \theta}\right)^2 d\theta
=\int_0^{2\pi} \frac{\partial F}{\partial \theta} \frac{\partial^2 F}{\partial \theta \partial t} d\theta
   =-\int_0^{2\pi} \frac{\partial^2 F}{\partial \theta^2} \frac{\partial F}{\partial t} d\theta
\\
&=& -\int_0^{2\pi} \frac{\partial F}{\partial t}
    \left(\frac{1}{F^\prime (\kappa)\kappa^2}\frac{\partial F}{\partial t}-F+\lambda\right)d\theta
\\
&=& -\int_0^{2\pi} \frac{1}{F^\prime (\kappa)\kappa^2} \left(\frac{\partial F}{\partial t}\right)^2 d\theta
    +\frac{1}{2}\frac{d}{dt} \int_0^{2\pi} F^2 d\theta -\lambda \frac{d}{dt} \int_0^{2\pi}F d\theta.
\end{eqnarray*}
Since $\lambda = \frac{1}{2\pi}\int_0^{2\pi}F d\theta$, one gets that
\begin{eqnarray*}
\frac{d}{dt}\frac{1}{2}\int_0^{2\pi} \left(\frac{\partial F}{\partial \theta}\right)^2 d\theta
&\leq& \frac{1}{2}\frac{d}{dt} \int_0^{2\pi} F^2 d\theta -\frac{1}{4\pi} \frac{d}{dt}
\left(\int_0^{2\pi}F d\theta\right)^2,
\end{eqnarray*}
which gives us by integration
\begin{eqnarray}
\frac{1}{2}\int_0^{2\pi} \left(\frac{\partial F}{\partial \theta}\right)^2 d\theta
&\leq& \frac{1}{2}\int_0^{2\pi} \left(\frac{\partial F_0}{\partial \theta}\right)^2 d\theta
      +\frac{1}{2}\int_0^{2\pi} F^2 d\theta -\frac{1}{2}\int_0^{2\pi} (F_0)^2 d\theta
\nonumber\\
&&    -\frac{1}{4\pi}\left(\int_0^{2\pi}F d\theta\right)^2+\frac{1}{4\pi}\left(\int_0^{2\pi}F_0 d\theta\right)^2
\nonumber\\
&\leq& \frac{1}{2}\int_0^{2\pi} \left(\frac{\partial F_0}{\partial \theta}\right)^2 d\theta
      +\pi F(M)^2+\frac{1}{4\pi}\left(\int_0^{2\pi}F_0 d\theta\right)^2.\label{eq:3.2.201811}
\end{eqnarray}
Therefore, the function $F(\kappa(\theta, t))$ is equicontinuous with respect to $\theta$.
By Lemma \ref{lem:2.2.201811} and Lemma \ref{lem:2.4.201811}, the curvature has uniform bounds
\begin{eqnarray*}
0 < \kappa(\theta, t) \leq M, ~~(\theta, t)\in [0, 2\pi]\times [0, +\infty).
\end{eqnarray*}
So does the function $F(\kappa(\theta, t))$.
Therefore, the Arzel\`{a}-Ascoli Theorem tells us that there is a convergent subsequence $\{F(\kappa(\theta, t_i))\}$, where $t_i$ tends to infinity.

\begin{lemma}\label{lem:3.2.201811}
There is a positive constant $c_0$ independent of time such that
\begin{eqnarray}\label{eq:3.3.201811}
\lambda (t)\geq c_0>0, \ \ \ \ \ \ t\geq 0.
\end{eqnarray}
\end{lemma}
\begin{proof}
%
By the positivity of $\kappa$ and $F(u)$ for all $u>0$,
one can conclude that $\lambda (t)>0$ if $t>0$. Suppose there is a sequence $t_i$ tending to $+\infty$ such that
$\lambda (t_i)\rightarrow 0$. Choose a subsequence $\{\widetilde{t}_i\}$ of $\{t_i\}$ so that $\{F(\kappa(\theta, \widetilde{t}_i))\}$
converges uniformly to a continuous function as $\widetilde{t}_i \rightarrow +\infty$.
If one sets $$F_\infty (\theta): = \lim_{\widetilde{t}_i \rightarrow +\infty} F(\kappa(\theta, \widetilde{t}_i)),$$ then
\begin{eqnarray*}
0= \lim_{\widetilde{t}_i \rightarrow +\infty} \lambda (\widetilde{t}_i)
 = \lim_{\widetilde{t}_i \rightarrow +\infty} \frac{1}{2\pi} \int_0^{2\pi}F(\kappa(\theta, \widetilde{t}_i)) d\theta
 = \frac{1}{2\pi} \int_0^{2\pi}F_\infty (\theta) d\theta.
\end{eqnarray*}
Using the positivity of $F(\kappa)$ again, one has $F_\infty(\theta)\equiv 0$.
Now, fix $\theta$. Since $\kappa (\theta, t)$ is bounded with respect to $t$, $\{\kappa (\theta, t)|t\geq 0\}$ has a convergent
subsequence. Denote by $\kappa_\infty (\theta)$ the limit of a convergent subsequence $\{\kappa(\theta, t_i)\}$
where $t_i\rightarrow +\infty$. Since $F$ is continuous,
$$F(\kappa_\infty (\theta)) =\lim_{t_i\rightarrow +\infty}F(\kappa(\theta, t_i))=F_\infty(\theta)= 0.$$
Noticing $F(0)=0$, $F(u)>0$ (for all $u>0$) and $F$ is increasing, one can conclude that $\kappa_\infty (\theta)= 0$.
Its every convergent subsequence tends to the same limit 0, so $\kappa (\theta, t)$ converges to 0 as $t\rightarrow +\infty$, for every fixed $\theta$.
By the arbitrariness of $\theta$,
$$\lim_{t\rightarrow +\infty} \kappa (\theta, t)\equiv 0$$
in the sense of $\theta$-pointwise. Therefore, the length of the curve
\begin{eqnarray*}
L(t)=\int_0^{2\pi} \frac{1}{\kappa(\theta, t)} d\theta
\end{eqnarray*}
is unbounded as $t \rightarrow +\infty$. This is impossible,
because the length of the evolving curve is fixed under the flow. So the function $\lambda(t)$ also has a positive lower bound on
$[0, +\infty)$.
\end{proof}

\begin{lemma}\label{lem:3.3.201811}
There is a positive constant $m$ independent of time such that
\begin{eqnarray}\label{eq:3.4.201811}
\kappa (\theta, t)\geq m
\end{eqnarray}
holds for all $(\theta, t) \in [0, 2\pi] \times [0, +\infty)$.
\end{lemma}
\begin{proof}
Consider the time interval $[nT_1, (n+1)T_1]$ ($n=0, 1, 2, \cdots$), where $T_1$ is a constant independent of time and defined
in Section 2. Choose an original point $O_n$ so that the support function satisfies
(\ref{eq:2.8.201811}) on the domain $Q_n := [0, 2\pi] \times [nT_1, (n+1)T_1]$. Let $\Delta$ be a positive constant such that $\Delta\geq 2p$
on $Q_n$. For example,
let $\Delta =L$. Set $\psi := \frac{F}{\Delta - p}$. Since
\begin{eqnarray*}
\frac{\partial \psi}{\partial \theta} = \frac{1}{\Delta - p}\frac{\partial F}{\partial \theta}
   +\frac{F}{(\Delta - p)^2}\frac{\partial p}{\partial \theta},
\end{eqnarray*}
and
\begin{eqnarray*}
\frac{\partial^2 \psi}{\partial \theta^2} = \frac{1}{\Delta - p}\frac{\partial^2 F}{\partial \theta^2}
   +\frac{2}{(\Delta - p)^2}\frac{\partial F}{\partial \theta}\frac{\partial p}{\partial \theta}
   +\frac{F}{(\Delta - p)^2}\frac{\partial^2 p}{\partial \theta^2}
   +\frac{2F}{(\Delta - p)^3}\left(\frac{\partial p}{\partial \theta}\right)^2,
\end{eqnarray*}
one has the evolution equation of $\psi$:
\begin{eqnarray*}
\frac{\partial \psi}{\partial t}
&=& \frac{1}{\Delta - p} \frac{\partial F}{\partial t} + \frac{F}{(\Delta - p)^2} \frac{\partial p}{\partial t}
\\
&=& \frac{F^\prime \kappa^2}{\Delta - p} \left(\frac{\partial^2 F}{\partial \theta^2} +F-\lambda(t) \right) + \frac{F}{(\Delta - p)^2} (\lambda(t) -F)
\\
&=& F^\prime \kappa^2 \left[\frac{\partial^2 \psi}{\partial \theta^2}
   -\frac{2}{(\Delta - p)^2}\frac{\partial F}{\partial \theta}\frac{\partial p}{\partial \theta}
   -\frac{F}{(\Delta - p)^2}\frac{\partial^2 p}{\partial \theta^2}
   -\frac{2F}{(\Delta - p)^3}\left(\frac{\partial p}{\partial \theta}\right)^2\right]
   \\
   && +\frac{F^\prime \kappa^2}{\Delta - p}(F-\lambda(t)) + \frac{F}{(\Delta - p)^2} (\lambda(t) -F)
\\
&=& F^\prime \kappa^2 \frac{\partial^2 \psi}{\partial \theta^2}
      -\frac{2F^\prime \kappa^2}{\Delta - p}\frac{\partial p}{\partial \theta}\frac{\partial \psi}{\partial \theta}
      -\frac{F^\prime \kappa^2 \psi}{\Delta - p}\frac{\partial^2 p}{\partial \theta^2}
      \\
&&      +\frac{\Delta F^\prime \kappa^2 \psi}{\Delta - p} - \frac{p F^\prime \kappa^2 \psi}{\Delta - p}
      -\frac{\lambda F^\prime \kappa^2}{\Delta - p}+\frac{\lambda F}{(\Delta - p)^2}-\psi^2.
\end{eqnarray*}
Using the formula $\frac{\partial^2 p}{\partial \theta^2}+p=\frac{1}{\kappa}$, one can continue to compute on the domain $Q_n$,
\begin{eqnarray*}
\frac{\partial \psi}{\partial t}
&=& F^\prime \kappa^2 \frac{\partial^2 \psi}{\partial \theta^2}
      -\frac{2F^\prime \kappa^2}{\Delta - p}\frac{\partial p}{\partial \theta}\frac{\partial \psi}{\partial \theta}
     +\frac{\Delta F^\prime \kappa^2 \psi}{\Delta - p} - \frac{F^\prime \kappa \psi}{\Delta - p}
      -\frac{\lambda F^\prime \kappa^2}{\Delta - p}+\frac{\lambda F}{(\Delta - p)^2}-\psi^2
\\
&=&   F^\prime \kappa^2 \frac{\partial^2 \psi}{\partial \theta^2}
      -\frac{2F^\prime \kappa^2}{\Delta - p}\frac{\partial p}{\partial \theta}\frac{\partial \psi}{\partial \theta}
      +\frac{\Delta F^\prime \kappa^2 \psi}{\Delta - p}
\\
&& +\frac{\lambda F}{\Delta - p}\left(\frac{1}{2(\Delta - p)}-\frac{F^\prime \kappa^2}{F}\right)
   +\left(\frac{\lambda}{2(\Delta - p)}-\frac{F^\prime \kappa}{\Delta - p}-\psi\right)\psi.
\end{eqnarray*}

Since $\lim\limits_{u\rightarrow 0^+} \frac{F^\prime (u)\cdot u^2}{F(u)}=0$ and $\frac{1}{L} < \frac{1}{\Delta - p}$, there exists $u_1>0$ such that if $u\in (0, u_1)$
then
\begin{eqnarray*}
\frac{F^\prime (u)\cdot u^2}{F(u)} <\frac{1}{2L}< \frac{1}{2(\Delta - p)}.
\end{eqnarray*}
Since $\lim\limits_{u\rightarrow 0^+} F^\prime (u)\cdot u=0$, there exists $u_2>0$ such that if $u\in (0, u_2)$ then by (\ref{eq:2.8.201811}) and (\ref{eq:3.3.201811}),
\begin{eqnarray*}
F^\prime (u)\cdot u < \frac{\delta c_0}{4L} < \frac{\delta\lambda}{4(\Delta - p)}.
\end{eqnarray*}

If $\psi(\theta, t) < \frac{F(u_1)}{L}$, i.e. $F(\kappa) < \frac{\Delta -p}{L}F(u_1) < F(u_1)$, then the monotonicity of $F$ implies $\kappa < u_1$. And
by the existence of $u_1$, one gets
\begin{eqnarray*}
\frac{1}{2(\Delta - p)} - \frac{F^\prime (\kappa)\cdot \kappa^2}{F(\kappa)} > 0.
\end{eqnarray*}
If $\psi < \min\{\frac{c_0}{4L}, \frac{F(u_2)}{L}\}$, then
$$\psi < \frac{c_0}{4L} <\frac{\lambda}{4(\Delta - p)}$$
and
$$F(\kappa) < \frac{\Delta -p}{L}F(u_2) < F(u_2).$$
Since $F$ is increasing, one has $\kappa < u_2$. And by (\ref{eq:2.7.201811}) and the existence of $u_2$, one may estimate
$$\frac{F^\prime (\kappa)\cdot \kappa}{\Delta -p} < \frac{F^\prime (\kappa)\cdot \kappa}{\delta}< \frac{\lambda}{4(\Delta - p)}.$$
Now, if $\psi < \min\{\frac{c_0}{4L}, \frac{F(u_2)}{L}\}$ then
\begin{eqnarray*}
\frac{\lambda}{2(\Delta - p)}-\frac{F^\prime \kappa}{\Delta - p}-\varphi>0.
\end{eqnarray*}
Therefore, the function $\psi_{\min}(t): = \min\{\psi (\theta, t)|\theta \in [0, 2\pi]\}$
is increasing once
\begin{eqnarray*}
\psi(\theta, t) < \min\left\{\frac{c_0}{4L}, ~\frac{F(u_1)}{L}, ~\frac{F(u_2)}{L}\right\},
\end{eqnarray*}
where the constants $c_0, u_1$ and $u_2$ are independent of time.
Applying the maximum principle to the evolution equation of $\psi$, one gets
\begin{eqnarray*}
\psi(\theta, t) \geq \psi_{\min} (t) \geq
\min\left\{\min_{\theta\in [0, 2\pi]} \psi(\theta, 0), ~\frac{c_0}{4L}, ~\frac{F(u_1)}{L}, ~\frac{F(u_2)}{L}\right\}:=\psi_0>0
\end{eqnarray*}
on the domain $Q_n$. It follows from the Mathematical induction that $\psi(\theta, t)$ has lower positive bound $\psi_0$ for
$(\theta, t) \in [0, 2\pi] \times [0, +\infty)$.

Thus, under the flow (\ref{eq:2.1.201811}), the function $F$ has a uniformly lower positive bound
$$F(\kappa) = (\Delta - p)\psi \geq p\psi_{\min}(t) \geq 2\delta\psi_0>0,~~(\theta, t) \in Q_n, n=0, 1, 2, \cdots,$$
where $\Delta =L$, $\delta$ and $\psi_0$ are positive constants independent of time.
Since $F(0) =0$ and $F(u)$ is strictly increasing, there is a positive constant $m$ which is independent of time such that $\kappa \geq m$ for
all $(\theta, t)\in [0, 2\pi]\times [0, +\infty)$. The proof is completed.
\end{proof}

Lemma \ref{lem:3.3.201811} tells us that the evolution equation of $\kappa$ does not degenerate as $t\rightarrow +\infty$. Then uniform bounds of
$\kappa ~(0< m \leq \kappa \leq M)$ imply that there exists a constant $C_n$ independent of time and $\theta$ such that
\begin{eqnarray}\label{eq:3.5.201811}
\left|\frac{d^nF}{du^n} (\kappa)\right|\leq C_n, ~~~~n=1, 2, \cdots,
\end{eqnarray}
under the flow (\ref{eq:2.1.201811}). There is also a positive constant $c_1$ independent of time such that
\begin{eqnarray}\label{eq:3.6.201811}
F^\prime (\kappa(\theta, t)) \geq c_1
\end{eqnarray}
for all $(\theta, t)\in [0, 2\pi] \times [0, +\infty)$.

Using (\ref{eq:3.2.201811}) and (\ref{eq:3.6.201811}), one has, on the domain $[0, 2\pi] \times [0, +\infty)$,
\begin{eqnarray}\label{eq:3.7.201811}
\int_0^{2\pi} \left(\frac{\partial \kappa}{\partial \theta}\right)^2 d\theta &=& \int_0^{2\pi} \left(\frac{1}{F^{\prime}}\frac{\partial F}{\partial \theta}\right)^2 d\theta
\nonumber\\
&\leq & \frac{1}{c_1^2} \left[\int_0^{2\pi} \left(\frac{\partial F_0}{\partial \theta}\right)^2 d\theta + 4\pi F(M)^2  \right].
\end{eqnarray}
So $\kappa(\theta, t)$ is equicontinuous with respect to $\theta$.

The curvature has uniformly upper and lower bounds and it is equicontinuous,
so there is a convergent subsequence $\kappa (\theta, t_i)$ as $t_i\rightarrow +\infty$. Furthermore,
one can show that the curvature converges to a constant as $t\rightarrow +\infty$.

\begin{lemma}\label{lem:3.4.201811}
Under the flow (\ref{eq:2.1.201811}), the curvature of the evolving curve has a limit:
\begin{eqnarray}\label{eq:3.8.201811}
\lim_{t\rightarrow +\infty} \kappa (\theta, t) = \frac{2\pi}{L}.
\end{eqnarray}
\end{lemma}
\begin{proof}
By the evolution equations of $A$  and $F$, one can compute
\begin{eqnarray*}
\frac{d^2A}{dt^2} &=& \frac{d}{dt} \left[-\int_0^{2\pi} \frac{F}{k} d\theta+ L\lambda\right]
\\
&=& -\int_0^{2\pi} F^\prime \kappa \left(\frac{\partial^2 F}{\partial \theta^2}+F-\lambda(t)\right)d\theta
   +\int_0^{2\pi} F\left(\frac{\partial^2 F}{\partial \theta^2}+F-\lambda(t)\right)d\theta
\\
&& +\frac{L}{2\pi} \int_0^{2\pi} F^\prime \kappa^2 \left(\frac{\partial^2 F}{\partial \theta^2}+F-\lambda(t)\right)d\theta
\\
&=& \int_0^{2\pi} \frac{F^{\prime\prime}}{F^\prime} \kappa \left(\frac{\partial F}{\partial \theta}\right)^2 d\theta
     +\int_0^{2\pi} \left(\frac{\partial F}{\partial \theta}\right)^2  d\theta
     -\int_0^{2\pi} F^{\prime} \kappa F  d\theta
     +\lambda \int_0^{2\pi} F^{\prime} \kappa d\theta
\\
&&   -\int_0^{2\pi} \left(\frac{\partial F}{\partial \theta}\right)^2  d\theta
     +\int_0^{2\pi} F^2 d\theta - \lambda \int_0^{2\pi} F d\theta
\\
&&
    -\frac{L}{2\pi} \int_0^{2\pi} \kappa^2 \frac{F^{\prime\prime}}{F^\prime}
            \left(\frac{\partial F}{\partial \theta}\right)^2 d\theta
    - \frac{L}{\pi} \int_0^{2\pi} \left(\frac{\partial F}{\partial \theta}\right)^2 \kappa d\theta
\\
&&
    + \frac{L}{2\pi} \int_0^{2\pi} F^{\prime} \kappa^2 F  d\theta
    - \frac{\lambda L}{2\pi} \int_0^{2\pi} F^{\prime} \kappa^2 d\theta
\\
&\leq&  \int_0^{2\pi} \frac{|F^{\prime\prime}|}{F^\prime} \kappa \left(\frac{\partial F}{\partial \theta}\right)^2 d\theta
        +\lambda \int_0^{2\pi} F^{\prime} \kappa d\theta
        +\int_0^{2\pi} F^2 d\theta
\\
&&    +\frac{L}{2\pi} \int_0^{2\pi} \kappa^2 \frac{|F^{\prime\prime}|}{F^\prime}
            \left(\frac{\partial F}{\partial \theta}\right)^2 d\theta
    + \frac{L}{2\pi} \int_0^{2\pi} F^{\prime} \kappa^2 F  d\theta.
\end{eqnarray*}
Using upper bounds of $F, |F^{\prime\prime}|, F^\prime$ and $\kappa$ and the positive lower bound of $F^\prime$,
one may bound $\frac{d^2A}{dt^2}$ as follows.
\begin{eqnarray*}
\frac{d^2A}{dt^2} &\leq&  \frac{C_2}{c_1}M\int_0^{2\pi} \left(\frac{\partial F}{\partial \theta}\right)^2 d\theta
        +2\pi F(M) C_1 M + 2\pi F(M)^2
\\
&&    +\frac{L}{2\pi}M^2 \frac{C_2}{c_1} \int_0^{2\pi} \left(\frac{\partial F}{\partial \theta}\right)^2 d\theta
    + LC_1M^2 F(M).
\end{eqnarray*}
By (\ref{eq:3.2.201811}), $\frac{d^2A}{dt^2}$ has an upper bound independent of time and $\theta$.

Noticing that $\frac{dA}{dt}$ is nonnegative (see Lemma \ref{lem:2.3.201811}) and
\begin{eqnarray*}
\int_0^{+\infty} \frac{dA}{dt} dt = \lim_{t\rightarrow \infty} A(t) -A_0 \leq \frac{L^2}{4\pi},
\end{eqnarray*}
one has a limit
\begin{eqnarray}\label{eq:3.9.201811}
\lim_{t\rightarrow \infty} \frac{dA}{dt} =0.
\end{eqnarray}
Substituting a convergent subsequence $\kappa (\theta, t_i)$ into the evolution equation of $A$ yields
\begin{eqnarray*}
\frac{dA}{dt} (t_i)= -\int_0^{2\pi} \frac{F(\kappa(\theta, t_i))}{\kappa(\theta, t_i)}d\theta
      + \frac{L}{2\pi}\int_0^{2\pi} F(\kappa(\theta, t_i))d\theta.
\end{eqnarray*}
Denote by $\kappa_\infty (\theta)$ the limit of the $\kappa (\theta, t_i)$ as $t_i \rightarrow +\infty$. Let $t_i \rightarrow +\infty$,
then one can get
\begin{eqnarray*}
0 = -\int_0^{2\pi} \frac{F(\kappa_\infty(\theta))}{\kappa_\infty(\theta)}d\theta
    +\frac{L}{2\pi}\int_0^{2\pi} F(\kappa_\infty(\theta))d\theta.
\end{eqnarray*}
It follows from the equality case in Andrews' inequality that $\kappa_\infty (\theta)$ is a constant.
Since the flow (\ref{eq:2.1.201811}) preserves the length, this constant is $2\pi/ L$, i.e.,
\begin{eqnarray*}
\kappa_\infty (\theta) = \frac{2\pi}{L}.
\end{eqnarray*}

It is shown that all convergent subsequence of $\kappa (\theta, t)$ tend to the same limit $\frac{2\pi}{L}$,
therefore,  the curvature function also has the same limit. The limit (\ref{eq:3.8.201811}) is proved.
\end{proof}

According (\ref{eq:3.4.201811}) and (\ref{eq:3.6.201811}), the evolution equation of $\kappa$ is uniformly parabolic with smooth coefficients.
Set $a_{11} := \kappa^2F^{\prime}(\kappa)$. By the uniform bounds of $\kappa$ (see the equations (\ref{eq:3.4.201811}) and (\ref{eq:2.11.201811})),
there exist positive constants $\beta_1, \beta_2$ independent of time such that
$$1 \leq  \beta_2 a_{11} x^2$$
for all $|x| \geq \beta_1$. By Lieberman's local gradient estimate (Theorem 11.18 in \cite{Lieberman-1996}) and the fact
that $F(\kappa(\theta, t))$ is periodic with respect to $\theta$, there exists a constant $M_1$ independent of time
such that
\begin{eqnarray}\label{eq:3.10.201811}
\left|\frac{\partial F}{\partial \theta}(\theta, t)\right| \leq \widetilde{M}_1,~~ (\theta, t)\in [0, 2\pi] \times (0, +\infty).
\end{eqnarray}
For the sake of completeness, we present a new proof of the gradient estimate \eqref{eq:3.10.201811} in Appendix of this paper, by revising the estimates in the
Subsection \ref{sec:2.4.201811}. Comparing to Lieberman's theory, the method used there is relatively elementary.

Using the regularity theory of parabolic equations or the simpler method by Lin-Tsai \cite{Lin-Tsai-2012}, there exist constants $\widetilde{M}_i$ independent of time such that
\begin{eqnarray}\label{eq:3.11.201811}
\left|\frac{\partial^i F}{\partial \theta^i}(\theta, t)\right| \leq \widetilde{M}_i,~~ (\theta, t)\in [0, 2\pi] \times (0, +\infty), ~~i=2, 3, \cdots.
\end{eqnarray}
The proof of \eqref{eq:3.11.201811} is a routine repetition of the proof of Lemma \ref{lem:2.8.201811}. We omit the details here.
By the uniform bounds of $\kappa$ and \eqref{eq:3.5.201811}-\eqref{eq:3.6.201811}, there exist constants $M_i$ independent of time such that
\begin{eqnarray}\label{eq:3.12.201811}
\left|\frac{\partial^i \kappa}{\partial \theta^i}(\theta, t)\right| \leq M_i, ~~(\theta, t)\in [0, 2\pi] \times (0, +\infty), ~~i=1, 2, 3, \cdots.
\end{eqnarray}
From (\ref{eq:2.17.201811}), (\ref{eq:3.10.201811}), (\ref{eq:3.11.201811}), (\ref{eq:3.12.201811}) and the limit (\ref{eq:3.8.201811}), one can conclude
\begin{eqnarray}\label{eq:3.13.201811}
\lim_{t\rightarrow +\infty}\frac{\partial^i \kappa}{\partial \theta^i}(\theta, t)
=\lim_{t\rightarrow +\infty}\frac{\partial^i F}{\partial \theta^i}(\theta, t) =0, ~~i=1, 2, \cdots.
\end{eqnarray}
This is the $C^\infty$ convergence of the curvature.

\section{Convergence of the evolving curve}\label{sec:4.201811}
Until now, the evolving curve has been proved asymptotically circular by the limit (\ref{eq:3.8.201811}). In order to prove that the flow (\ref{eq:2.1.201811})
can efficiently evolve $X(\cdot, t)$ into a circle, one needs to show that the evolving curve indeed has a limit as time
goes to infinity. In fact, if the speed of the flow (\ref{eq:2.1.201811}) exponentially decays
then the evolving curve $X(\cdot, t)$ converges to a fixed limiting curve as time goes to infinity.

By the evolution equation of $F$ (see (\ref{eq:2.8.201811})), one may compute
\begin{eqnarray*}
&& \frac{d}{dt}\frac{1}{2}\int_0^{2\pi} \left(\frac{\partial F}{\partial \theta}\right)^2 d\theta
\ =\ -\int_0^{2\pi} \frac{\partial^2 F}{\partial \theta^2} \frac{\partial F}{\partial t} d\theta
\\
&=& -\int_0^{2\pi} \frac{\partial^2 F}{\partial \theta^2} F^\prime \kappa^2
     \left[\frac{\partial^2 F}{\partial \theta^2} +F-\lambda(t) \right] d\theta
\\
&=& -\int_0^{2\pi} F^\prime \kappa^2 \left(\frac{\partial^2 F}{\partial \theta^2}\right)^2 d\theta
    +\int_0^{2\pi} F^\prime \kappa^2 \left(\frac{\partial F}{\partial \theta}\right)^2 d\theta
    \\
   && +\int_0^{2\pi} \frac{F^{\prime\prime}}{F^\prime} \kappa^2 (F-\lambda)\left(\frac{\partial F}{\partial \theta}\right)^2 d\theta
   + 2\int_0^{2\pi} \kappa (F-\lambda)\left(\frac{\partial F}{\partial \theta}\right)^2 d\theta.
\end{eqnarray*}
The limit (\ref{eq:3.8.201811}) says that the curvature converges to $2\pi/L$. So one obtians
\begin{eqnarray*}
\lim_{t\rightarrow +\infty} F^\prime(\kappa)\cdot \kappa^2 = F^\prime\left(\frac{2\pi}{L}\right)\cdot \left(\frac{2\pi}{L}\right)^2
\end{eqnarray*}
and
\begin{eqnarray*}
\lim_{t\rightarrow +\infty} (F(\kappa) - \lambda(t))= \lim_{t\rightarrow +\infty} \left(F(\kappa) - \frac{1}{2\pi}\int_0^{2\pi} F(\kappa) d \theta \right)= 0.
\end{eqnarray*}
For any $\varepsilon >0$, there exists $T_0>0$ such that, if $t>T_0$ then $|F-\lambda| < \varepsilon$ and
\begin{eqnarray*}
F^\prime \left(\frac{2\pi}{L}\right)\cdot \left(\frac{2\pi}{L}\right)^2 -\varepsilon \leq F^\prime(\kappa)\cdot \kappa^2 \leq
F^\prime \left(\frac{2\pi}{L}\right)\cdot \left(\frac{2\pi}{L}\right)^2 +\varepsilon.
\end{eqnarray*}
So, one has the following estimate for large $t>T_0$,
\begin{eqnarray} \label{eq:4.1.201811}
\frac{d}{dt}\frac{1}{2}\int_0^{2\pi} \left(\frac{\partial F}{\partial \theta}\right)^2 d\theta
&\leq& -\left(F^\prime \left(\frac{2\pi}{L}\right)\cdot \left(\frac{2\pi}{L}\right)^2 -\varepsilon \right)
         \int_0^{2\pi} \left(\frac{\partial^2 F}{\partial \theta^2}\right)^2 d\theta
   \nonumber\\
   &&
    +\left(F^\prime \left(\frac{2\pi}{L}\right)\cdot \left(\frac{2\pi}{L}\right)^2 + \varepsilon \right)
         \int_0^{2\pi} \left(\frac{\partial F}{\partial \theta}\right)^2 d\theta
   \nonumber\\
   && +\left(\frac{C_2}{c_1}M^2 \varepsilon +2M\varepsilon \right)
      \int_0^{2\pi} \left(\frac{\partial F}{\partial \theta}\right)^2 d\theta.
\end{eqnarray}
Since $$\frac{\partial F}{\partial \theta}=  F^\prime (\kappa)\frac{\partial}{\partial \theta}\left(\frac{1}{\rho}\right)
=  F^\prime (\kappa)\left(-\frac{1}{\rho^2}\frac{\partial \rho}{\partial \theta}\right)$$
and
\begin{eqnarray*}
\frac{\partial^2 F}{\partial \theta^2} &=& F^{\prime\prime} (\kappa)\left(\frac{1}{\rho^2}\frac{\partial \rho}{\partial \theta}\right)^2
  -F^\prime (\kappa) \frac{1}{\rho^2}\frac{\partial^2 \rho}{\partial \theta^2}
  +F^\prime (\kappa) \frac{2}{\rho^3} \left(\frac{\partial \rho}{\partial \theta}\right)^2
\\
&=& \frac{1}{\rho^3}\left(\frac{F^{\prime\prime}}{\rho} +2 F^\prime\right)\left(\frac{\partial \rho}{\partial \theta}\right)^2
  -F^\prime (\kappa) \frac{1}{\rho^2}\frac{\partial^2 \rho}{\partial \theta^2},
\end{eqnarray*}
one may compute
\begin{eqnarray*}
\int_0^{2\pi} \left(\frac{\partial^2 F}{\partial \theta^2}\right)^2 d\theta
&=&  \int_0^{2\pi} \frac{1}{\rho^6}\left(\frac{F^{\prime\prime}}{\rho} +2 F^\prime\right)^2 \left(\frac{\partial \rho}{\partial \theta}\right)^4 d\theta
    + \int_0^{2\pi} (F^\prime)^2 \frac{1}{\rho^4} \left(\frac{\partial^2 \rho}{\partial \theta^2}\right)^2 d\theta
\\
&& -2 \int_0^{2\pi} F^\prime \frac{1}{\rho^5} \left(\frac{F^{\prime\prime}}{\rho} +2 F^\prime\right)
   \left(\frac{\partial \rho}{\partial \theta}\right)^2\frac{\partial^2 \rho}{\partial \theta^2}d\theta.
\end{eqnarray*}
It follows from the convergence of $\kappa$ that one has the convergence
\begin{eqnarray*}
\lim\limits_{t \rightarrow\infty} \frac{\partial \rho}{\partial \theta}
= \lim\limits_{t \rightarrow\infty} \frac{\partial^2 \rho}{\partial \theta^2} = 0.
\end{eqnarray*}
So, there exists $\widetilde{\varepsilon} \in (0, \varepsilon)$ such that, at large time,
\begin{eqnarray*}
\int_0^{2\pi} \left(\frac{\partial^2 F}{\partial \theta^2}\right)^2 d\theta
&\geq&  \int_0^{2\pi} \left[\left(F^\prime\left(\frac{2\pi}{L}\right) \right)^2\left(\frac{2\pi}{L}\right)^4  -\widetilde{\varepsilon} \right]
\left(\frac{\partial^2 \rho}{\partial \theta^2}\right)^2 d\theta
\\
&& -\widetilde{\varepsilon} \int_0^{2\pi} \left(\frac{\partial \rho}{\partial \theta}\right)^2 d\theta.
\end{eqnarray*}
Using next Wirtinger type inequality,
\begin{eqnarray*}
\int_0^{2\pi} \left(\frac{\partial^2 \rho}{\partial \theta^2}\right)^2 d\theta
\geq 4 \int_0^{2\pi} \left(\frac{\partial \rho}{\partial \theta}\right)^2 d\theta
\end{eqnarray*}
and considering at large time, one could estimate
\begin{eqnarray*}
\int_0^{2\pi} \left(\frac{\partial^2 F}{\partial \theta^2}\right)^2 d\theta
&\geq& 4\int_0^{2\pi} \left[\left(F^\prime\left(\frac{2\pi}{L}\right) \right)^2\left(\frac{2\pi}{L}\right)^4  - \frac{5\widetilde{\varepsilon}}{4}\right]
      \left(\frac{\partial \rho}{\partial \theta}\right)^2 d\theta
\\
&\geq& 4 \int_0^{2\pi} \left[(F^\prime(\kappa))^2 \frac{1}{\rho^4}  - 2\widetilde{\varepsilon}\right]
      \int_0^{2\pi} \left(\frac{\partial \rho}{\partial \theta}\right)^2 d\theta.
\end{eqnarray*}
Since $(F^\prime(\kappa))^2 \frac{1}{\rho^4}$ has uniform lower and upper bounds independent of time, one may choose $\widetilde{\varepsilon}$
small enough such that
\begin{eqnarray} \label{eq:4.2.201811}
\int_0^{2\pi} \left(\frac{\partial^2 F}{\partial \theta^2}\right)^2 d\theta
\geq 2 \int_0^{2\pi} (F^\prime(\kappa))^2 \frac{1}{\rho^4} \left(\frac{\partial \rho}{\partial \theta}\right)^2 d\theta
= 2\int_0^{2\pi} \left(\frac{\partial F}{\partial \theta}\right)^2 d\theta.
\end{eqnarray}
Substituting (\ref{eq:4.2.201811}) into (\ref{eq:4.1.201811}), one obtains
\begin{eqnarray*}
\frac{d}{dt}\frac{1}{2}\int_0^{2\pi} \left(\frac{\partial F}{\partial \theta}\right)^2 d\theta
\leq \left(-F^\prime\left(\frac{2\pi}{L}\right)\left(\frac{2\pi}{L}\right)^2 +3\varepsilon
        +\frac{C_2}{c_1}M^2 \varepsilon +2M\varepsilon\right)
       \int_0^{2\pi} \left(\frac{\partial F}{\partial \theta}\right)^2 d\theta.
\end{eqnarray*}

Let $\varepsilon>0$ be small enough such that
$3\varepsilon +\frac{C_2}{c_1}M^2 \varepsilon +2M\varepsilon
  < \frac{1}{2}F^\prime(\frac{2\pi}{L}) \left(\frac{2\pi}{L}\right)^2$. Then
one gets for large $t$,
\begin{eqnarray} \label{eq:4.3.201811}
\frac{d}{dt}\int_0^{2\pi} \left(\frac{\partial F}{\partial \theta}\right)^2 d\theta
\leq - \frac{1}{2} F^\prime \left(\frac{2\pi}{L} \right)\cdot \left(\frac{2\pi}{L}\right)^2
       \int_0^{2\pi} \left(\frac{\partial F}{\partial \theta}\right)^2 d\theta.
\end{eqnarray}
Thus $\int_0^{2\pi} \left(\frac{\partial F}{\partial \theta}\right)^2 d\theta$
decays exponentially.

\begin{lemma}\label{lem:4.1.201811}
If $t$ is large enough then the speed of the flow (\ref{eq:2.1.201811}) decays exponentially.
\end{lemma}
\begin{proof}
It follows from the classical Sobolev's inequality, one has
\begin{eqnarray}\label{eq:4.4.201811}
|F-\lambda| \leq \frac{1}{\sqrt{2\pi}}\left(\int_0^{2\pi} \left( F-\lambda \right)^2 d\theta\right)^{\frac{1}{2}}
 + \sqrt{2\pi} \left(\int_0^{2\pi} \left(\frac{\partial F}{\partial \theta}\right)^2 d\theta\right)^{\frac{1}{2}}.
\end{eqnarray}
Since $\int_0^{2\pi} (F- \lambda) d\theta =0$, Wirtinger's inequality implies that
\begin{eqnarray}\label{eq:4.5.201811}
\int_0^{2\pi} (F- \lambda)^2 d\theta \leq \int_0^{2\pi} \left(\frac{\partial F}{\partial \theta}\right)^2 d\theta.
\end{eqnarray}
Substituting (\ref{eq:4.5.201811}) into (\ref{eq:4.4.201811}) yields
\begin{eqnarray*}
|F-\lambda| \leq \left(\frac{1}{\sqrt{2\pi}} + \sqrt{2\pi}\right) \left(\int_0^{2\pi} \left(\frac{\partial F}{\partial \theta}\right)^2 d\theta\right)^{\frac{1}{2}}.
\end{eqnarray*}
Since $\int_0^{2\pi} \left(\frac{\partial F}{\partial \theta}\right)^2 d\theta$
exponentially decays when $t > T_0$, one has shown that $|F-\lambda|$ also has exponential decay.
\end{proof}

\begin{theorem}\label{thm:4.2.201811}
Under the flow (\ref{eq:2.1.201811}), the evolving curve converges to a limiting circle as time
goes to infinity.
\end{theorem}
\begin{proof}
Since the speed $|F-\lambda|$ of the flow decays exponentially if $t$ is large enough, one could claim that the limit
$\int_0^{t} (F- \lambda) d\tau$ exists as $t \rightarrow + \infty$. Therefore, integrating the flow equation, one gets the the limit of the evolving curves as follows
\begin{eqnarray}\label{eq:4.6.201811}
\lim_{t \rightarrow + \infty} \widetilde{X}(\cdot, t) = X_0(\cdot) +  \int_0^{+\infty} (F- \lambda) dt : = X_{\infty}(\cdot),
\end{eqnarray}
where $X_0$ is the initial curve of the flow. Since the flow preserves the length of the evolving curve and the curvature $\kappa(\cdot, t)$ converges to the
constant $\frac{2\pi}{L}$, the limiting curve $X_{\infty}$ is a circle with radius $\frac{L}{2\pi}$.
\end{proof}

Until now, we have proved that both the curvature and the evolving curve converge as $t\rightarrow +\infty$.
So the proof of Theorem \ref{thm:1.1.201811} is a combination of Theorem \ref{thm:2.5.201811},
Lemma \ref{lem:3.4.201811} and Theorem \ref{thm:4.2.201811}. Once the speed of the flow is proved decaying exponentially, the $C^{\infty}$ convergence
of the flow is a straightforward computation as did in the papers \cite{Gage-Hamilton, Lin-Tsai-2008}. For the sake of brevity, the related details are omitted.

\section{Appendix: uniform gradient estimate}\label{sec:5.201811}

Let $X(\cdot, t): [0, 2\pi] \times [0, + \infty) \rightarrow \mathbb{R}^2$ be a global solution to the flow \eqref{eq:2.1.201811} with convex
initial curve $X_0$. By \eqref{eq:2.11.201811} and \eqref{eq:3.4.201811}, the curvature $\kappa (\cdot, t)$ has lower and upper
bounds independent of time, i.e.,
$$0< m \leq \kappa \leq M, ~~~~(\theta, t)\in [0, 2\pi] \times [0, + \infty).$$
Recalling \eqref{eq:3.5.201811} and \eqref{eq:3.6.201811}, one has bounds of $F^{\prime}(\kappa)$ as follows
$$0< c_1 \leq F^{\prime}(\kappa) \leq C_1, ~~~~(\theta, t)\in [0, 2\pi] \times [0, + \infty).$$
In this appendix, we show that there exists a constant $\widetilde{M}_1$ independent of time such that
\begin{eqnarray*}
\left|\frac{\partial F}{\partial \theta}(\theta, t)\right| \leq \widetilde{M}_1, ~~(\theta, t)\in [0, 2\pi] \times (0, +\infty).
\end{eqnarray*}
Comparing to Lieberman's theory of gradient estimate \cite{Lieberman-1996}, the method used here is more elementary.

\begin{lemma}\label{lem:5.1.201811}
Let $p \geq 2$ be an integer. Under the flow (\ref{eq:2.1.201811}), we have the estimate
\begin{eqnarray}
\frac{d}{dt}\frac{1}{2p} \int_0^{2\pi} \left(\frac{\partial F}{\partial \theta}\right)^{2p} d\theta
&\leq &  - \frac{(2p-1) c_1 m^2}{2p^2} \int_0^{2\pi} \left(\frac{\partial F}{\partial \theta}\right)^{2p} d\theta
\nonumber \\
&& + \frac{(2p-1) c_1 m^2}{4 \pi p^2} \left[\int_0^{2\pi} \left(\frac{\partial F}{\partial \theta}\right)^{p} d\theta \right]^2
\nonumber \\
&& +\frac{(2p-1) C_1 M^2 F(M)^2}{2} \int_0^{2\pi} \left(\frac{\partial F}{\partial \theta}\right)^{2p-2} d\theta. \label{eq:5.1.201811}
\end{eqnarray}
\end{lemma}
\begin{proof}
We compute by the evolution equation of $F(\kappa (\theta, t))$ (see \eqref{eq:2.9.201811}),
\begin{eqnarray*}
\frac{d}{dt}\frac{1}{2p}\int_0^{2\pi} \left(\frac{\partial F}{\partial \theta}\right)^{2p} d\theta
&=&\int_0^{2\pi} \left(\frac{\partial F}{\partial \theta}\right)^{2p-1} \frac{\partial^2 F}{\partial \theta \partial t} d\theta
\\
&=& -(2p-1)\int_0^{2\pi} \left(\frac{\partial F}{\partial \theta}\right)^{2p-2}
   \frac{\partial^2 F}{\partial \theta^2} \frac{\partial F}{\partial t} d\theta
\\
&=& -(2p-1)\int_0^{2\pi}\left(\frac{\partial F}{\partial \theta}\right)^{2p-2}
       \frac{\partial^2 F}{\partial \theta^2} F^\prime \kappa^2
     \left[\frac{\partial^2 F}{\partial \theta^2} +F-\lambda(t) \right] d\theta
\\
&=& -(2p-1)\int_0^{2\pi}F^\prime \kappa^2 \left(\frac{\partial F}{\partial \theta}\right)^{2p-2}
       \left(\frac{\partial^2 F}{\partial \theta^2}\right)^2 d\theta
   \\
   &&-(2p-1)\int_0^{2\pi}F^\prime \kappa^2 (F-\lambda(t)) \left(\frac{\partial F}{\partial \theta}\right)^{2p-2}
       \frac{\partial^2 F}{\partial \theta^2} d\theta.
\end{eqnarray*}
Applying Cauchy-Schwarz Inequality, one gets
\begin{eqnarray}
\frac{d}{dt}\frac{1}{2p}\int_0^{2\pi} \left(\frac{\partial F}{\partial \theta}\right)^{2p} d\theta
&\leq&  -\frac{2p-1}{2} \int_0^{2\pi}F^\prime \kappa^2 \left(\frac{\partial F}{\partial \theta}\right)^{2p-2}
       \left(\frac{\partial^2 F}{\partial \theta^2}\right)^2 d\theta
\nonumber   \\
   && +\frac{(2p-1)}{2} \int_0^{2\pi}F^{\prime} \kappa^2 (F-\lambda(t))^2 \left(\frac{\partial F}{\partial \theta}\right)^{2p-2} d\theta. \label{eq:5.2.201811}
\end{eqnarray}

Using the lower bounds of $\kappa$ and $F^{\prime}(\kappa)$, one may estimate the first integral in the R.H.S. of \eqref{eq:5.2.201811} as
\begin{eqnarray*}
&& -\frac{2p-1}{2} \int_0^{2\pi}F^\prime \kappa^2 \left(\frac{\partial F}{\partial \theta}\right)^{2p-2}
       \left(\frac{\partial^2 F}{\partial \theta^2}\right)^2 d\theta
\\
&& ~~= -\frac{2p-1}{2p^2} \int_0^{2\pi} F^\prime \kappa^2
\left[ \frac{\partial }{\partial \theta}\left(\frac{\partial F}{\partial \theta}\right)^p \right]^2  d\theta
\\
&& ~~ \leq- \frac{(2p-1) c_1 m^2}{2p^2} \int_0^{2\pi}
\left[ \frac{\partial }{\partial \theta}\left(\frac{\partial F}{\partial \theta}\right)^p \right]^2  d\theta.
\end{eqnarray*}

Let $f=f(\theta)$ be a periodic and $C^1$ function on $[0, 2\pi]$. We have another Wirtinger's inequality for the function $f$:
\begin{eqnarray}\label{eq:5.3.201811}
\int_0^{2\pi} \left(\frac{\partial f}{\partial \theta}\right)^{2} d\theta
&\geq&  \int_0^{2\pi} f^{2} d\theta - \frac{1}{2\pi} \left(\int_0^{2\pi} f d\theta\right)^{2}.
\end{eqnarray}
By using \eqref{eq:5.3.201811}, one may continue to estimate
\begin{eqnarray}
&& -\frac{2p-1}{2} \int_0^{2\pi}F^\prime \kappa^2 \left(\frac{\partial F}{\partial \theta}\right)^{2p-2}
       \left(\frac{\partial^2 F}{\partial \theta^2}\right)^2 d\theta
\nonumber \\
&& ~~ \leq- \frac{(2p-1) c_1 m^2}{2p^2} \int_0^{2\pi} \left(\frac{\partial F}{\partial \theta}\right)^{2p}  d\theta
\nonumber \\
&& ~~~~+ \frac{(2p-1) c_1 m^2}{4\pi p^2}
    \left[\int_0^{2\pi} \left(\frac{\partial F}{\partial \theta}\right)^p d\theta\right]^2. \label{eq:5.4.201811}
\end{eqnarray}
Substituting \eqref{eq:5.3.201811} and \eqref{eq:5.4.201811} into \eqref{eq:5.2.201811} yields \eqref{eq:5.1.201811}.
\end{proof}

\begin{lemma}\label{lem:5.2.201811}
Both $\int_0^{2\pi} \left(\frac{\partial F}{\partial \theta}\right)^4 d\theta$ and $\int_0^{2\pi} \left(\frac{\partial F}{\partial \theta}\right)^6 d\theta$
have uniform upper bounds independent of time.
\end{lemma}
\begin{proof}
Set $p=2$ in the equation \eqref{eq:5.1.201811}. Then one has
\begin{eqnarray}
\frac{d}{dt}\frac{1}{4} \int_0^{2\pi} \left(\frac{\partial F}{\partial \theta}\right)^{4} d\theta
&\leq &  - \frac{3 c_1 m^2}{8} \int_0^{2\pi} \left(\frac{\partial F}{\partial \theta}\right)^{4} d\theta
\nonumber \\
&& + \frac{3 c_1 m^2}{16 \pi} \left[\int_0^{2\pi} \left(\frac{\partial F}{\partial \theta}\right)^2 d\theta \right]^2
\nonumber \\
&& +\frac{3 C_1 M^2 F(M)^2}{2} \int_0^{2\pi} \left(\frac{\partial F}{\partial \theta}\right)^2 d\theta. \label{eq:5.5.201811}
\end{eqnarray}
Since $\int_0^{2\pi} \left(\frac{\partial F}{\partial \theta}\right)^2 d\theta$ is bounded uniformly on time interval $[0, +\infty)$
(see \eqref{eq:3.2.201811}), the inequality \eqref{eq:5.5.201811} implies that
\begin{eqnarray*}
\frac{d}{dt}\frac{1}{4} \int_0^{2\pi} \left(\frac{\partial F}{\partial \theta}\right)^{4} d\theta
\leq   - \frac{3 c_1 m^2}{8} \int_0^{2\pi} \left(\frac{\partial F}{\partial \theta}\right)^{4} d\theta + A_0,
\end{eqnarray*}
where $A_0$ is a constant independent of time. Integrating this inequality gives us
\begin{eqnarray*}
\int_0^{2\pi} \left(\frac{\partial F}{\partial \theta}\right)^{4} d\theta \leq  \int_0^{2\pi} \left(\frac{\partial F_0}{\partial \theta}\right)^{4} d\theta
+ \frac{8A_0}{3 c_1 m^2}.
\end{eqnarray*}
The integral $\int_0^{2\pi} \left(\frac{\partial F}{\partial \theta}\right)^4 d\theta$ has a uniformly upper bound independent of time.

Set $p=3$ in the equation \eqref{eq:5.1.201811}. One gets
\begin{eqnarray}
\frac{d}{dt}\frac{1}{6} \int_0^{2\pi} \left(\frac{\partial F}{\partial \theta}\right)^6 d\theta
&\leq &  - \frac{5 c_1 m^2}{18} \int_0^{2\pi} \left(\frac{\partial F}{\partial \theta}\right)^6 d\theta
\nonumber \\
&& + \frac{5 c_1 m^2}{36 \pi} \left[\int_0^{2\pi} \left(\frac{\partial F}{\partial \theta}\right)^3 d\theta \right]^2
\nonumber \\
&& +\frac{5 C_1 M^2 F(M)^2}{2} \int_0^{2\pi} \left(\frac{\partial F}{\partial \theta}\right)^4 d\theta. \label{eq:5.6.201811}
\end{eqnarray}
By Holder's inequality, one may estimate
\begin{eqnarray*}
\int_0^{2\pi} \left(\frac{\partial F}{\partial \theta}\right)^3 d\theta
\leq
(2\pi)^{1/4} \left[\int_0^{2\pi} \left(\frac{\partial F}{\partial \theta}\right)^4 d\theta \right]^{3/4}.
\end{eqnarray*}
So the integral $\int_0^{2\pi} \left(\frac{\partial F}{\partial \theta}\right)^3 d\theta$ also has a uniformly upper bound independent of time.
Now \eqref{eq:5.6.201811} implies
\begin{eqnarray}
\frac{d}{dt}\frac{1}{6} \int_0^{2\pi} \left(\frac{\partial F}{\partial \theta}\right)^6 d\theta
\leq  - \frac{5 c_1 m^2}{18} \int_0^{2\pi} \left(\frac{\partial F}{\partial \theta}\right)^6 d\theta + B_0,
\end{eqnarray}
where $B_0$ is a constant independent of time. Integrating this inequality gives us a uniformly upper bound of
the integral $\int_0^{2\pi} \left(\frac{\partial F}{\partial \theta}\right)^6 d\theta$.
\end{proof}

\begin{theorem}\label{thm:5.3.201811}
Under the global flow \eqref{eq:2.1.201811}, the derivative $|\frac{\partial F(\kappa (\theta, t))}{\partial \theta}|$ has an upper bound on the domain $[0, 2\pi] \times (0, +\infty)$.
\end{theorem}
\begin{proof}
Recalling in the Subsection \ref{sec:2.4.201811}, we have shown that, for small $\varepsilon >0$,
\begin{eqnarray*}
\frac{d}{dt}\frac{1}{2}\int_0^{2\pi} \left(\frac{\partial^2 F}{\partial \theta^2}\right)^2 d\theta
& \leq & \int_0^{2\pi} \left(A_1\varepsilon - F^\prime \kappa^2\right)\left(\frac{\partial^3 F}{\partial \theta^3}\right)^2 d\theta
 +A_2 \int_0^{2\pi}\left(\frac{\partial F}{\partial \theta}\right)^6 d\theta
\\
&& + A_3 \int_0^{2\pi} \left(\frac{\partial F}{\partial \theta}\right)^2 d\theta,
\end{eqnarray*}
where $A_1= A_1(c_1, C_1, C_2), A_2= A_2(\varepsilon, c_1, C_1, C_2)$ and $A_3= A_3(\varepsilon, c_1, C_1, C_2)$ are constants independent of time.

Applying Wirtinger's inequality \eqref{eq:5.3.201811}, one gets
\begin{eqnarray*}
\int_0^{2\pi} \left(\frac{\partial^3 F}{\partial \theta^3}\right)^2 d\theta
\geq \int_0^{2\pi} \left(\frac{\partial^2 F}{\partial \theta^2}\right)^2 d\theta.
\end{eqnarray*}
Choosing $\varepsilon$ small enough, one obtains from Lemma \ref{lem:5.2.201811},
\begin{eqnarray*}
\frac{d}{dt}\frac{1}{2}\int_0^{2\pi} \left(\frac{\partial^2 F}{\partial \theta^2}\right)^2 d\theta
& \leq & - \frac{c_1 m^2}{2} \int_0^{2\pi} \left(\frac{\partial^2 F}{\partial \theta^2}\right)^2 d\theta  +B_1,
\end{eqnarray*}
where $B_1$ is a constant independent of time. Integrating this inequality yields an upper bound of
$\int_0^{2\pi} \left(\frac{\partial^2 F}{\partial \theta^2}\right)^2 d\theta$ on time interval $(0, +\infty)$.

Since both the integrals $\int_0^{2\pi} \left(\frac{\partial F}{\partial \theta}\right)^2 d\theta$ and $\int_0^{2\pi} \left(\frac{\partial^2 F}{\partial \theta^2}\right)^2 d\theta$
are bounded for $t\in [0, +\infty)$, Sobolev's inequality tells us that the derivative
$|\frac{\partial F}{\partial \theta}|$ is also uniformly bounded.
\end{proof}

~\\
\textbf{Acknowledgments}
Laiyuan Gao is supported by National Natural Science Foundation of China (No.11801230).
Shengliang Pan is supported by National Natural Science Foundation of China (No.11671298 and No.12071347).
Laiyuan Gao would like to thank Professor Zhuan Ye at Jiangsu Normal University for the inspirational talks with him.
And Gao also thanks professors Michael Gage, Richard Hamilton, Dong-Ho Tsai and Xiaoliu Wang, because he learnt a tremendous amount
during past many years, both from them and from their papers, such as \cite{Gage-1986, Gage-Hamilton, Wang-Tsai-2015} and so on.

{\bf Laiyuan Gao}

School of Mathematics and Statistics, Jiangsu Normal University.

No.101, Shanghai Road, Xuzhou 221116, Jiangsu, China.

Email: lygao@jsnu.edu.cn\\

{\bf Shengliang Pan}

School of Mathematical Sciences, Tongji University.

No.1239, Siping Road, Yangpu 200092, Shanghai, China.

Email: slpan@tongji.edu.cn \\

\end{document}